%% file: comput.tex
\documentclass{llncs}

\usepackage[utf8]{inputenc}

\input{macros}

\pgfplotsset{compat=1.17}
\begin{document}

\renewcommand\algorithmicrequire{\textbf{Input:}}
\renewcommand\algorithmicensure{\textbf{Output:}}
\newcommand\Require{\REQUIRE}
\newcommand\Return{\RETURN}
\newcommand\Ensure{\ENSURE}
\renewcommand{\emph}[1]{\textit{#1}}
\renewcommand{\em}[1]{\textit{#1}}

\title{Computation of Hilbert class polynomials and modular polynomials from supersingular elliptic curves}
\author{Antonin Leroux}
\institute{DGA-MI, Bruz, France 
 \\ IRMAR - UMR 6625, Université de Rennes, France
\\ \email{antonin.leroux@polytechnique.org}}

\pagestyle{plain}

\maketitle

\begin{abstract}
    We present several new heuristic algorithms to compute class polynomials and modular polynomials modulo a prime $p$ by revisiting the idea of working with supersingular elliptic curves. 
    The best known algorithms to this date are based on ordinary curves, due to the supposed inefficiency of the supersingular case. While this was true a decade ago, the recent advances in the study of supersingular curves through the Deuring correspondence motivated by isogeny-based cryptography has provided all the tools to perform the necessary tasks efficiently.
    
    Our main ingredients are two new heuristic algorithms to compute the $j$-invariants of supersingular curves having an endomorphism ring contained in some set of isomorphism class of maximal orders. The first one is derived easily from the existing tools of isogeny-based cryptography, while the second introduces new ideas to perform that task efficiently for a big number of maximal orders at the same time.  

    For each of the polynomials (Hilbert and modular), we obtain two algorithms. The first one, that we will qualify as \textit{direct}, is based on the computation of a set of well-chosen supersingular $j$-invariants defined over $\FF_{p^2}$ and uses the aforementioned algorithm to translate maximal orders to $j$-invariants as its main building block.
    The second one is a CRT algorithm that applies the direct algorithm on a set of small primes and reconstruct the result modulo $p$ with the chinese remainder theorem. 
    
    In both cases, the direct algorithm achieves the best known complexity for primes $p$ that are relatively small compared to the discriminant (for the Hilbert case) and to the level (for the modular case). 
    
    Our CRT algorithms matches the complexities of the state-of-the-art CRT approach based on ordinary curves, while improving some of the steps, thus opening the possibility to a better practical efficiency.

\end{abstract}


\section{Introduction}

Hilbert class polynomials and modular polynomials are central objects in number theory, and their computation have numerous applications. One field where these computations are of particular interest is cryptography. The main applications are to be found in elliptic curve cryptography and pairing-based cryptography, but we can also mention, more marginally, the recent field of isogeny-based cryptography. 

Class polynomials, for instance, play a central role in the CM method, which is the main approach to find ordinary curves with a prescribed number of points over a given finite field (see \cite{atkin1993elliptic,broker2007efficient}). This has applications to primality proving with the ECPP method and finding pairing friendly-curves with the Cocks-Pinch method.  

Modular polynomials are related to isogenies between elliptic curves. Historically, they play a very important role in the SEA point counting algorithm \cite{elkies1998elliptic,schoof1995counting} which remains one of the main algorithms used in elliptic-curve cryptography to generate cryptograhic curves. 
Moreover, the interest in isogenies have been renewed with the rise of isogeny-based cryptography. While most applications tend to use the more efficient Vélu formulas \cite{V71}, we can cite a few instances where modular polynomials have been considered. For example, it is used in the CRS key exchange \cite{couveignes,RS06}, the very first isogeny-based protocol, and we can also mention the OSIDH construction \cite{CK19}.     

The goal of this work is to explore theoretical and practical improvements to the best-known algorithms to compute class polynomials and modular polynomials modulo prime numbers through the use of supersingular curves.



\paragraph*{Related work.} 
One of the main problems behind the computation of class polynomials and modular polynomials is the huge size of their coefficients over $\ZZ$. There exists several algorithms of quasi-linear complexity \cite{enge2009complexity,couveignes2002action,sutherland2011computing,broker2012modular}, but more often than not, memory is the real bottleneck in the concrete computations of those polynomials. 
In theory, size is less an issue when the result is needed modulo some prime $p$, but this is only true in practice if we have a way to skip entirely the computation over $\ZZ$, which is not so easy to get. 

Nonetheless, Sutherland \cite{sutherland2011computing} proved that this could be done for class polynomials by a careful application of the CRT method. The result was later applied to modular polynomials by Bröker, Lauter and Sutherland (BLS) \cite{broker2012modular}. 
The main advantage of the CRT method compared to other approaches is the low memory requirement (almost optimal in the size of the final output), and this is why this method has achieved the best practical results.

For both class and modular polynomials, the main tools used in the CRT algorithms from \cite{sutherland2011computing,sutherland2012accelerating,broker2012modular} are /ordinary elliptic curves.

Supersingular curves had been considered in the context of the Hilbert polynomial computation in \cite{belding2008computing} and in the context of the modular polynomials computation in \cite{charles2005computing}, but these methods were discarded over the time because they were slower than those based on ordinary curves.   


The situation has changed with the recent interest on the connection between supersingular curves and quaternion algebras sparked by isogeny-based cryptography.

Since the work of Deuring \cite{D41}, it is known that endomorphism rings of supersingular curves in characteristic $p$ are isomorphic to maximal orders in the quaternion algebra $\QA$ ramified at $p$ and infinity, and that, conversely, every such maximal order types arises in this way. This is the first result of what is now called the \textit{Deuring correspondence}.
In this work, we are particularly interested in the task of computing the $j$-invariants of the (at most 2) supersingular elliptic curves over $\FF_{p^2}$ having a given maximal order type as endomorphism ring.

The first concrete effort to realize that task is an algorithm of Cervino \cite{cervino2004correspondence} to compute the endomorphism rings of all supersingular curves in characteristic $p$. The complexity of this algorithm is $O(p^{2 + \varepsilon})$ and it becomes rapidly impractical. This algorithm was more recently improved by Chevyrev and Galbraith in \cite{chevyrev2014constructing} but the complexity is still  $O(p^{1,5 + \varepsilon})$. 
As part of cryptanalytic efforts to understand the difficulty of various problems related to the Deuring correspondence, a heuristic algorithm of polynomial complexity in $\log(p)$ was introduced by Eisenträger, Hallgren, Lauter, Morrison and Petit \cite{EHLMP18} . This algorithm builds upon the previous works of Kohel, Lauter, Petit and Tignol \cite{KLPT14} and Galbraith, Petit and Silva \cite{GPS17}. 

More concretely, these works prove that an isogeny can be efficiently computed between two supersingular curves of known endomorphism ring by translating the problem over the quaternions with the Deuring correspondence, solving the translated problems over quaternions, before translating back the solution as an isogeny. This can be applied directly to compute the $j$-invariants of all curves with an endomorphism ring contained in a given maximal order type by using one starting curve $E_0$ of known endomorphism ring (such a curve can always be computed efficiently with the CM method).

Recently, in \cite{robert2022some}, Robert introduced new algorithms to compute modular polynomials from supersingular curves, this time using the high-dimensional isogeny techniques recently introduced in the context of isogeny-based cryptography cryptanalysis.  
Like us, Robert has a direct algorithm and a CRT algorithm. The direct algorithm has a quadratic complexity in the level $\ell$, just as our direct algorithm, but the asymptotic complexity is somewhat differently balanced.
The CRT algorithm reaches the same complexity has Sutherland's. 
Robert algorithm have the very nice advantage to have a proven complexity without any heuristic, but they require more memory, in particular the CRT algorithm.

\paragraph*{Contributions.}

Our main contribution is to revisit previous methods to compute Hilbert and modular polynomials from supersingular curves with the recent algorithmic progress on the Deuring correspondence. 

Our main sub-routine aims at translating a set of isomorphism class of maximal orders into their corresponding supersingular $j$-invariants under the Deuring correspondence. We introduce two algorithms, with different performance profiles, to perform that task. 

With these new algorithms, we obtain an improvement over the asymptotic complexity of the class and modular polynomials computation in a wide range of primes below some upper-bounds that depend either on the discriminant of the class polynomial or the level of the modular polynomial. Moreover, we show that our new algorithm can also be used in the CRT method to reach the same complexity as ordinary curves, but with possibly better practical efficiency.

\subsection{Technical overview}
\label{sec: technical overview}

We start by looking at our main-subroutine that consists in the computation of the $j$-invariants of supersingular elliptic curves corresponding to some set of maximal order isomorphism classes (called maximal order types, see \cref{def: type}). In the rest of this article, unless specified otherwise, a curve is considered to be a supersingular elliptic curve.   

\paragraph*{Maximal orders to $j$-invariants.}
We propose two algorithms dedicated to that task. Let consider that a set $\frakS$ of types is given in input, together with some prime $p$. 

Our first algorithm is called $\OrdersToJSmall{}$ and it consists merely in a sequential execution of a sub-algorithm from \cite{EHLMP18} that performs the desired translation for one type of maximal order. When everything is done carefully, it can be executed in $O(\log(p)^{5 + \varepsilon})$ under experimentally verified heuristics detailed in \cite{KLPT14} and related to the probability for numbers represented by some quadratic form to be prime. Thus, since $\OrdersToJSmall{}$ consists in $\# \frakS$ executions of this sub-algorithm, the total heuristic complexity of \OrdersToJSmall{} is $O(\# \frakS \log(p)^{5 + \varepsilon})$. For a generic $p$ and set of maximal order type $\frakS$, we do not know how to do better than that. However, when $\frakS$ contains almost all maximal order types (the maximal size being upper-bounded by the number of supersingular curves), it becomes sub-optimal due to the amount of redundant computation performed along the way. In that case, it becomes much more practical to use an algorithm designed to sieve through the entire set of types, only focusing on the ones in $\frakS$ when they're met along the way. It requires a bit of care to perform this task in the most efficient manner but it can be done, and this leads to the algorithm $\OrdersToJBig{}$ of complexity $O(\#\frakS \log(p)^{2 + \varepsilon} + p \log(p)^{1 + \varepsilon})$. This algorithm requires one heuristic that we detail in \cref{sec: order to jinv computation}, as \cref{claim: number of maximal orders}. It is related to the expansion property of the supersingular isogeny graphs.   

We stress that both algorithms are designed to work (and analyzed) for a generic prime $p$ which is why they are so interesting. 

\paragraph*{The direct algorithms.} Our  heuristic algorithms \OrdersToJSmall{} and \OrdersToJBig{} can be used directly to compute the roots of class and modular polynomials modulo $p$ (under the assumption that these roots are supersingular). The method is pretty straightforward: find the maximal order types corresponding to the desired roots, then, compute them with either \OrdersToJSmall{} or \OrdersToJBig{}. 
With the complexity we have stated, this is already enough to obtain an asymptotic improvement over existing generic methods  when $p$ is not too big (compared to the discriminant or level of the associated class or modular polynomial).

If we write $S$ the "degree" of the polynomial (it is $h(D)= O(\sqrt{|D|} \log(D)^\varepsilon)$ for Hilbert polynomials of discriminant $|D|$ and $O(\ell^2)$ for modular polynomials of level $\ell$), then we obtain the following complexity with \OrdersToJSmall{}: 
$$ O( S \log^{5 + \varepsilon} p + S \log^{2 + \varepsilon} S  \log p).$$ 

With \OrdersToJBig{}, the complexity becomes 
$$ O(S \log^{2 + \varepsilon} p + p \log^{1 + \varepsilon} p+  S \log^{2 + \varepsilon} S \log p).$$ 

In both cases, the latter term comes from the polynomial reconstruction step that must be performed to recover the polynomial from its roots. Note that the size of the output is $O(S \log p)$. In terms of space, the requirement is quasi-optimal in both cases: so $O(S^{1+ \varepsilon} \log p)$.  

It is clear that the second algorithm will be better when $p = O(S \log(S))$. However, whenever $S = o(p)$ (which is often the case in applications), it will be better to use the variant with \OrdersToJSmall{}. 





\paragraph*{The CRT for class polynomials.}
Let us take a prime $p$ and a discriminant $D <0$. We want an efficient algorithm to compute $H_D(X) \mod p$. 
Our main algorithm is essentially the same as the one introduced by Belding, Bröker, Enge and Lauter in \cite{belding2008computing}, and later improved by  Sutherland in \cite{sutherland2011computing} .

Let us write $\frakO$ for the quadratic imaginary order of discriminant $D < 0$. 
We may assume that the factorization of $D$ is known as computing it is negligible compared to the rest of the computation. 
We define $\primesO{}$ to be a set of primes. 
We write $B_D$ for the bound on the bit-size of the coefficients of $H_D$ over $\ZZ$. 

Here is how the algorithm works:
\begin{enumerate}
    \item Select some primes $p_1 ,\ldots,p_n$ in $\primesO$ with $\prod_{i=1}^n p_i > 2^{B_D}$.
    \item Compute a suitable representation of $\Cl(D)$.
    \item For each $p_i \in \primesO$:
    \begin{enumerate}
        \item Compute the coefficients of $H_D \mod p_i$. 
        \item Update CRT sums for each coefficient of $H_D$.
    \end{enumerate}
    \item Recover the coefficients of $H_D \mod p$. 
\end{enumerate}

The main difference between the several variants of the CRT approach (including ours) lies is in the choice of the set $\primesO$ and in the way to compute $H_D \mod p_i$. In \cite{belding2008computing}, different algorithms were proposed to handle the distinct cases of split and inert $p_i$. 
In both cases, the $H_D \mod p_i$ are constructed from their roots. These roots are always $j$-invariants of elliptic curves in characteristic $p_i$, but this is where the similarity ends. 
In the former case, the elliptic curves are ordinary and are defined over $\FF_{p_i}$, whereas in the latter case, we obtain supersingular elliptic curves defined over $\FF_{p_i^2}$. The ordinary and supersingular case are very different and the resulting algorithms are also very different. 

Sutherland \cite{sutherland2011computing} improved the method from \cite{belding2008computing} for split primes by a careful choice of the primes and other tricks to improve the computation of $H_D \mod p_i$. 


We improve the method presented in \cite{belding2008computing} for non-split primes by making use of the Deuring correspondence with our \OrdersToJBig{} algorithm to replace Cervino's algorithm \cite{cervino2004correspondence}. 


\paragraph*{The CRT for modular polynomials.} 
Let $\ell$ be a prime distinct from $p$. We want a CRT approach to compute  $\Phi_\ell(X,Y) \mod p$. It can be done in a very similar fashion to class polynomials by computing $\Phi_\ell \mod p_i$ for some $p_i$ in a set $\primesell{}$ and reconstruct the final polynomial via the CRT. 

Bröker, Lauter and Sutherland (BLS in the rest of this article) in \cite{broker2012modular} proposed to use primes of the form $(t^2 - 4 v^2 D \ell^2)$ with $t,v,D \in \NN$ for which there is a very specific volcano structure involving $\ell$-isogenies. This structure implies the existence of two distinct sets of ordinary curves defined over $\FF_{p_i}$: the curves with endomorphism ring isomorphic to $\frakO$ for some quadratic imaginary order $\frakO$ of discriminant $D$ and class number bigger than $\ell+2$ and the curves with endomorphism ring isomorphic to $\ZZ + \ell \frakO$. Since the latter are $\ell$-isogenous to the former, it is possible to recover the full $\Phi_\ell \mod p_i$ by computing the $j$-invariants corresponding to these two sets of curves. 
The volcano structure allows for efficient computation by minimizing the number of $\ell$-isogeny computations. 

For supersingular curves, the choice of primes is even easier than for class polynomials: we can use any primes $p_i$ that is big enough. As long as the number of supersingular curves is bigger than $\ell+2$ we will be able to recover the full modular polynomial. 
This idea has already been considered by Charles and Lauter in 2005 \cite{charles2005computing} but in a rather direct way (where each of the $\ell$-isogeny involved is computed using the Vélu formulas).  

We prove that using the Deuring correspondence and \OrdersToJBig{}, we can avoid entirely any $\ell$-isogeny computation and minimize the cost of elliptic curve operations.  

\paragraph*{Generic improvements to the CRT method.} There are several ways to improve the CRT method in practical applications. First, alternative class polynomials and modular functions (with smaller height bounds) can be used instead of the standard Hilbert class polynomial and modular polynomials for the same practical purpose. This was the focus of the paper \cite{enge2010class}.  

Second, for a number of applications such as the CM method and the SEA point counting algorithm, computing these polynomials is actually not necessary. What is really needed is the ability to evaluate them. Sutherland showed \cite{sutherland2012accelerating,sutherland2013evaluation} that it was possible to do better than compute-then-evaluate in both cases, providing, in particular, an additional improvement in terms of memory requirement for those applications. 

Using supersingular curves rather than ordinary ones should not prevent from applying all these practical improvements. For clarity's sake we focus on the simpler computation of the standard polynomials and leave to the reader the task of adapting these improvements to our new setting which should not be too daunting.

\paragraph*{Organisation of the article.} The rest of this paper is organized as follows: in \cref{sec: background}, we introduce some background on isogenies, quaternion algebras and the Deuring correspondence. Then, in \cref{sec: order to jinv computation}, we introduce our main new algorithm to compute efficiently $j$-invariants corresponding to  maximal order types. In \cref{sec: class poly computation}, we explain in details how this algorithm can be applied to the computation of class polynomial with the CRT method. 
In \cref{sec: modular poly computation}, we do the same for modular polynomials.

\paragraph*{Acknowledgement.} We thank Andrew Sutherland and anonymous reviewers for very useful feedback on this work.

\section{Background material}
\label{sec: background}



\subsection{Notations}
\label{sec: notations}

\paragraph*{Basic complexities.} 

We write $M_\ZZ(b)$ for the cost of multiplying two integers of less than $b$ bits. For asymptotic complexities we consider $M_\ZZ(b) = O(b^{1 + \varepsilon})$. For instance, this covers the complexity of all arithmetic operations in a finite field $\FF_p$ of characteristic $p$ of less than $b$ bits.

Similarly, we write $M_\PP(b)$ for the cost (in terms of arithmetic operations over $k$) of multiplying two polynomials of degree smaller than $b$ over a base field $k$. Depending on the size of $b$ we will either use $M_\PP(b) = O(b \log(b)^{1 + \varepsilon}$ or $O(b^{1 + \varepsilon})$.

Finally, the cost of fast interpolation algorithm for a polynomial of degree $b$ is $O(M(_\PP(b) \log(b))$.

\input{background.tex}

\section{Computing $j$-invariants corresponding to maximal orders.}
\label{sec: order to jinv computation}

\input{order_to_jinv.tex}

\section{Computation of the Hilbert class polynomial}
\label{sec: class poly computation}





In this section, we introduce our two algorithms to compute Hilbert class polynomial. In \cref{sec: mod class poly computation}, we present the direct algorithm, and in \cref{sec: class poly comparison}, we introduce the CRT algorithm and we also provide a comparison with Sutherland's CRT algorithm from \cite{sutherland2011computing}.

\input{hilbert_comput.tex}

\section{Computation of the modular polynomials}
\label{sec: modular poly computation}

In this section, we introduce our new algorithms to compute modular polynomials. 
In \cref{sec: direct supersingular modular poly computation}, we present the direct algorithm, and in \cref{sec: CRT modular poly computation}, we describe the CRT algorithm and compare the various existing algorithms of the literature.



\input{modular_comput.tex}

\section{Conclusion}

We have introduced several new algorithms to compute modular polynomials of level $\ell$ and Hilbert polynomials of discriminant $D$ modulo a generic prime number $p$ from supersingular curves. The direct version of our algorithmsn have complexity in $\ell^2$ and $\sqrt{|D|}$ respectively for generic primes. Depending on the relative size of $\ell$,$|D|$ and $p$, we exhibit improvements over the best known asymptotic complexities for a significant range of primes. 

Moreover, when applied to the CRT method, we obtain an algorithm whose complexity is the same as previously known CRT method, but with the potential to give a practical improvement (in particular in the case of modular polynomials).  

It remains to see how efficient our new algorithms are in practice. There are several practical challenges to overcome before providing an implementation of the proposed algorithms (in particular related to the field extensions involved in the computations of some isogenies), and this is why we leave the concrete implementation to future work.

\bibliographystyle{alpha}
\bibliography{biblio}

\appendix

\end{document}

%% file: macros.tex
\usepackage[centertags]{amsmath}
\usepackage{amsfonts}
\usepackage{amssymb}

\usepackage{amscd}
\usepackage{mathtools}
\usepackage{hyperref}
\usepackage{algorithmic, algorithm}
\usepackage{graphicx}
\usepackage{todonotes}
\usepackage{color}
\usepackage{pgfplots}
\usepackage{faktor}
\usepackage{mathrsfs} 
\usepackage[capitalize]{cleveref}
\usepackage{subfig}
\usepackage{tikz}
\usepackage[framemethod=tikz]{mdframed}
\usepackage{tabularx,booktabs}
\usetikzlibrary{shadows.blur, shapes.arrows}
\usepackage[shortlabels]{enumitem}
\numberwithin{equation}{section}
\usepackage{xspace}

\renewcommand{\contentsname}{Contents}

\def\hmath$#1${\texorpdfstring{{\rmfamily\textit{#1}}}{#1}}

\newcommand{\ie}{i.e.,\xspace}

\definecolor{green}{RGB}{34,139,34}

\newcommand{\OrdersToJBig}[1]{\hyperref[alg: orders to j-inv]{\textsf{\upshape Orders\-To\-j\-Invariant\-Big\-Set}\ensuremath{}}}
\newcommand{\OrdersToJSmall}{{\textsf{\upshape Orders\-To\-jInvariant\-Small\-Set}}}
\newcommand{\OrderToJ}[1]{{\textsf{\upshape Single\-Order\-To\-j\-Invariant}\ensuremath{(#1)}}}

\newcommand{\OrdersToJName}{\hyperref[alg: orders to j-inv]{\textsf{\upshape Orders\-To\-jInvariant\-Big\-Set}}}
\newcommand{\OrderToJName}{{\textsf{\upshape Single\-Order\-To\-jInvariant}}}

\newcommand{\KerToId}[1]{\hyperref[alg: kernel to ideal]{\textsf{\upshape Kernel\-To\-Ideal}\ensuremath{_{#1}}}}
\newcommand{\IdToKer}[1]{\hyperref[alg: ideal to kernel]{\textsf{\upshape Ideal\-To\-Kernel}\ensuremath{_{#1}}}}

\newcommand{\ItIE}[1]{\hyperref[alg: new algo]{\textsf{\upshape Ideal\-To\-Isogeny\-From\-Eichler}\ensuremath{_{#1}}}}
\newcommand{\IdToIsoEich}[1]{\hyperref[alg: new algo]{\textsf{\upshape Ideal\-To\-Isogeny\-From\-Eichler}\ensuremath{_{#1}}}}
\newcommand{\IdToIsoKLPT}[1]{\hyperref[alg:IdealToIsogeny2bullet]{\textsf{\upshape Ideal\-To\-Isogeny\-From\-KLPT}\ensuremath{_{#1}}}}
\newcommand{\IdToIsoCoprime}[1]{\hyperref[alg:SpecialIdealToIsogeny]{\textsf{\upshape Ideal\-To\-Isogeny\-Coprime}\ensuremath{_{#1}}}}
\newcommand{\IdToIsoSmallKLPT}[1]{\hyperref[alg:IdealToIsogeny2eepsilon]{\textsf{\upshape Ideal\-To\-Isogeny\-Small\-From\-KLPT}\ensuremath{_{#1}}}}
\newcommand{\IdToIsoSmallEich}[1]{\hyperref[alg: new sub-algo]{\textsf{\upshape Ideal\-To\-Isogeny\-Small\-From\-Eichler}\ensuremath{_{#1}}}}
\newcommand{\SpEicherN}[1]{\hyperref[alg: special eichler norm eq]{\textsf{\upshape Special\-Eichler\-Norm}\ensuremath{_{#1}}}}
\newcommand{\EichlerN}[1]{\hyperref[alg: eichler norm eq]{\textsf{\upshape Eichler\-Norm}\ensuremath{_{#1}}}}
\newcommand{\IdEichlerN}[1]{\hyperref[alg: eichler ideal norm equation]{\textsf{\upshape Ideal\-Eichler\-Norm}\ensuremath{_{#1}}}}
\newcommand{\SubEichlerN}[1]{\hyperref[alg: extended eichler]{\textsf{\upshape Suborder\-Eichler\-Norm}\ensuremath{_{#1}}}}
\newcommand{\IdSubEichlerN}[1]{\hyperref[alg: ideal suborder norm eq]{\textsf{\upshape Ideal\-Suborder\-Eichler\-Norm}\ensuremath{_{#1}}}}

\newcommand{\FRI}[1]{\hyperref[alg: repres integer full]{\textsf{\upshape Full\-Represent\-Integer}\ensuremath{_{#1}}}}
\newcommand{\RI}[1]{\hyperref[alg: repres integer]{\textsf{\upshape Represent\-Integer}\ensuremath{_{#1}}}}
\newcommand{\FSA}[1]{\hyperref[alg: full strong approx]{\textsf{\upshape Full\-Strong\-Approximation}\ensuremath{_{#1}}}}
\newcommand{\SA}[1]{\hyperref[alg:strong approx]{\textsf{\upshape Strong\-Approximation}\ensuremath{_{#1}}}}
\newcommand{\KLPT}[1]{\hyperref[alg:KLPT]{\textsf{\upshape KLPT}\ensuremath{_{#1}}}}
\newcommand{\GenKLPT}[1]{\hyperref[alg: generic klpt]{\textsf{\upshape Generic\-KLPT}\ensuremath{_{#1}}}}
\newcommand{\SignKLPT}[1]{\hyperref[alg: sign klpt]{\textsf{\upshape Signing\-KLPT}\ensuremath{_{#1}}}}

\newcommand{\SmoothGen}[1]{\hyperref[alg: smooth basis]{\textsf{\upshape Generating\-Family}\ensuremath{_{#1}}}}

\newcommand{\Compress}[1]{\hyperref[alg:compression]{\textsf{\upshape Compression}\ensuremath{_{#1}}}}
\newcommand{\Decompress}[1]{\hyperref[alg:decompression]{\textsf{\upshape Decompression}\ensuremath{_{#1}}}}

\newcommand{\EndoEval}[1]{\hyperref[alg: endo eval]{\textsf{\upshape Endomorphism\-Evaluation}\ensuremath{_{#1}}}}
\newcommand{\IdealEvaluation}[1]{\hyperref[alg: ideal evaluation]{\textsf{\upshape Ideal\-Evaluation}\ensuremath{_{#1}}}}
\newcommand{\SuborderEvaluation}[1]{\hyperref[alg: suborder evaluation]{\textsf{\upshape Suborder\-Evaluation}\ensuremath{_{#1}}}}
\newcommand{\VerifIdeal}[1]{\hyperref[alg: verif ideal proof]{\textsf{\upshape Ideal\-Verification}\ensuremath{_{#1}}}}

\newcommand{\VerifSuborder}[1]{\hyperref[alg: verif suborder proof]{\textsf{\upshape Suborder\-Verification}\ensuremath{_{#1}}}}
\newcommand{\IdealToSuborder}[1]{\hyperref[alg: ideal to suborder]{\textsf{\upshape Ideal\-To\-Suborder}\ensuremath{_{#1}}}}
\newcommand{\CheckTrace}[1]{\hyperref[alg: check trace]{\textsf{\upshape Check\-Trace}\ensuremath{_{#1}}}}
\newcommand{\InverseTrapdoor}[1]{\hyperref[alg:inv]{\textsf{\upshape Inverse\-Trapdoor}\ensuremath{_{#1}}}}

\newcommand{\pSIDHKeyGen}[1]{\hyperref[alg: keygen]{\textsf{\upshape pSIDH\-Key\-Gen}\ensuremath{_{#1}}}}
\newcommand{\pSIDHKeyEx}[1]{\hyperref[alg: key exch]{\textsf{\upshape pSIDH\-Key\-Exchange}\ensuremath{_{#1}}}}
\newcommand{\SetaQuadOrder}[1]{\hyperref[alg: keygen param]{\textsf{\upshape Seta\-Quadratic\-Order\-Gen}\ensuremath{_{#1}}}}
\newcommand{\SetaCurveGen}[1]{\hyperref[alg:randomized trapdoor curve gen]{\textsf{\upshape Seta\-Curve\-Gen}\ensuremath{_{#1}}}}

\newcommand{\SupersingularEvaluation}[1]{\hyperref[alg: supersingular evaluation]{\textsf{\upshape Super\-singular\-Eval\-ua\-tion}\ensuremath{_{#1}}}}
\newcommand{\SpecialSupersingularEvaluation}[1]{\hyperref[alg: supersingular evaluation]{\textsf{\upshape Special\-Super\-singular\-Eval\-ua\-tion}\ensuremath{_{#1}}}}
\newcommand{\ModularEvaluationBigLevel}[1]{\hyperref[alg: modular evaluation]{\textsf{\upshape Mo\-dular\-Eval\-ua\-tion\-Big\-Level}\ensuremath{_{#1}}}}
\newcommand{\ModularEvaluationBigChar}[1]{\hyperref[alg: modular evaluation]{\textsf{\upshape Mo\-dular\-Eval\-ua\-tion\-Big\-Char\-acteristic}\ensuremath{_{#1}}}}
\newcommand{\ModularEvaluation}[1]{\hyperref[alg: modular evaluation]{\textsf{\upshape Mo\-dular\-Eval\-ua\-tion}\ensuremath{_{#1}}}}
\newcommand{\ModularComputation}[1]{\hyperref[alg: modular computation]{\textsf{\upshape Mo\-dular\-Com\-puta\-tion}\ensuremath{_{#1}}}}

\newcommand{\EndoRing}[1]{\hyperref[alg: endo ring]{\textsf{\upshape Endo\-morphism\-Ring}\ensuremath{_{#1}}}}
\newcommand{\SupersingularModularComputation}[1]{\hyperref[alg: supersingular modular computation]{\textsf{\upshape Super\-singular\-Modular\-Compu\-ta\-tion}\ensuremath{_{#1}}}}

\newcommand{\SupersingularHilbertComputation}[1]{\hyperref[alg: supersingular hilbert computation]{\textsf{\upshape Super\-singular\-Hilbert\-Compu\-ta\-tion}\ensuremath{_{#1}}}}

\newcommand{\HilbertComputation}[1]{\hyperref[alg: hilbert computation]{\textsf{\upshape Hilbert\-Compu\-ta\-tion}\ensuremath{_{#1}}}}

\newcommand{\llog}{\textnormal{llog}}

\newcommand{\SDp}{\mathfrak{S}_D(p)}
\newcommand{\primes}[1]{\mathcal{P}_{D_{#1}}}
\newcommand{\primesell}{\mathcal{P}_\ell}
\newcommand{\primesO}{\mathcal{P}_\frakO}

\newcommand{\norm}[1]{\mathbf{n}(#1)}

\newcommand{\frakS}{\mathfrak{S}}

\renewcommand{\O}{\mathcal{O}}
\newcommand{\order}{\mathcal{O}}

\newcommand{\QA}{\mathcal{B}_{p,\infty}}

\newcommand{\MM}{\mathbb{M}}

\newcommand{\FF}{\mathbb{F}}
\newcommand{\ZZ}{\mathbb{Z}}
\newcommand{\NN}{\mathbb{N}}
\newcommand{\QQ}{\mathbb{Q}}
\newcommand{\PP}{\mathbb{P}}

\newcommand{\frakO}{\mathfrak{O}}

\newcommand{\Cl}{\textnormal{Cl}}

\DeclareMathOperator{\Typ}{Typ}

\newcommand{\G}[1] {\langle #1 \rangle}

\newcommand{\poly}[1]{ O(\mathsf{poly} (#1))}
\newcommand{\polylog}[1]{ O(\mathsf{poly} (\log (#1) ))}

\definecolor{light blue}{RGB}{0,102,204}

\def\emptyset{\{\}}

\DeclareFontFamily{OT1}{rsfs}{}
\DeclareFontShape{OT1}{rsfs}{n}{it}{<-> rsfs10}{}
\DeclareMathAlphabet{\mathscr}{OT1}{rsfs}{n}{it}

\newcommand{\cH}{\mathcal{H}}

\newcommand{\vphi}{\varphi}

\newcommand{\comment}[1]{}

\renewcommand{\O}{\mathcal{O}}

\DeclareMathOperator{\tr}{tr}

\DeclareMathOperator{\End}{End}

\newtheorem{claime}{Claim}

\crefname{figure}{Figure}{Figures}
\crefname{problem}{Problem}{Problems}
\crefname{proposition}{Proposition}{Propositions}
\crefname{algorithm}{Algorithm}{Algorithms}
\crefname{step}{Step}{Steps}
\crefname{lemma}{Lemma}{Lemmas}
\crefname{definition}{Definition}{Definitions}
\crefname{claime}{Claim}{Claims}

\crefname{assumption}{Assumption}{Assumptions}
\crefname{conjecture}{Conjecture}{Conjectures}

\newcounter{tasknumber}
\newcommand{\task}[2][]{%
  \addtocounter{tasknumber}{1}%
  \begin{center}%
  \framebox[1.1\width]{\begin{minipage}{0.9\textwidth}%
  \textbf{Task \arabic{tasknumber}} \textit{\if!#1(unassigned)!\else (#1)\fi}: {#2}%
  \end{minipage}}%
  \end{center}%
}

%% file: background.tex
\subsection{Elliptic curves, quaternion algebras and the Deuring correspondence}
\label{sec: prelim deuring}

More precise references on the topics covered in this section are: the book of Silverman \cite{S94} for elliptic curves and isogenies, the book of John Voight \cite{voight} on quaternion algebras and theoretical aspects of the Deuring correspondence, the thesis of Antonin Leroux \cite{leroux2022quaternion} for the algorithmic aspects of the Deuring correspondence. 

\paragraph*{Supersingular elliptic curves and isogenies.}

An \textit{isogeny} $\varphi : E_1 \rightarrow E_2$ is a non-constant morphism sending the identity of $E_1$ to that of $E_2$. The degree of an isogeny is its degree as a rational map (see \cite{S09} for more details). When the degree $\deg(\varphi)=d$ is coprime to $p$, the isogeny is necessarily \textit{separable} and  $d  = \# \ker \varphi$.  An isogeny is said to be cyclic when its kernel is a cyclic group. The Vélu formulas \cite{V71} can be used to compute any cyclic isogeny from its kernel.  
For any $\varphi : E_1 \rightarrow E_2$, there exists a unique dual isogeny $\hat{\varphi}: E_2 \rightarrow E_1$, satisfying $\varphi \circ \hat{\varphi} = [\deg(\varphi)]$. 

\paragraph*{Endomorphism ring.}
An isogeny from a curve $E$ to itself is an \textit{endomorphism}. The set $\End(E)$ of all endomorphisms of $E$ forms a ring under addition and composition.
For elliptic curves defined over a finite field $\FF_q$, $\End(E)$ is isomorphic either to an order of a quadratic imaginary field or a maximal order in a quaternion algebra. In the first case, the curve is said to be \textit{ordinary} and otherwise \textit{supersingular}. We focus on the supersingular case in this article.
Every supersingular elliptic curve defined over a field of characteristic $p$ admits an isomorphic model over $\mathbb{F}_{p^2}$. It implies that there only a finite number of isomorphism class of supersingular elliptic curves. The Frobenius over $\FF_p$ is the only inseparable isogeny between supersingular curves and it has degree $p$. We write $\pi : E \rightarrow E^p$. For any supersingular curve $E$, the property $\End(E) \cong \End(E^p)$ is satisfied but we have $E \cong E^p$ if and only if $E$ has an isomorphic model over $\FF_p$. 

\paragraph*{Quaternion algebras.}

For $a,b \in \QQ^\star$ we denote by $ H(a,b) = \QQ+i\QQ+j\QQ+k\QQ$ the quaternion algebra over $\QQ$ with basis $1,i,j,k$ such that $i^2=a$, $j^2=b$ and $k=ij=-ji$.
Every quaternion algebra has a canonical involution that sends an element $\alpha = a_1 + a_2 i + a_3 j +a_4 k$ to its conjugate $\overline{\alpha} =  a_1 - a_2 i - a_3 j - a_4 k$. We define the \textit{reduced trace} and the \textit{reduced norm} by $tr(\alpha) = \alpha + \overline{\alpha} \enspace$ and $ n(\alpha) = \alpha \overline{\alpha}$. 

\paragraph*{Orders and ideals.}
A \textit{fractional ideal} $I$ of a quaternion algebra $\mathcal{B}$ is a $\ZZ$-lattice of rank four contained in $\mathcal{B}$. We denote by $n(I)$ the \emph{norm} of $I$, defined as the $\ZZ$-module generated by the reduced norms of the elements of $I$.  

An order $\O$ is a subring of $\mathcal{B}$ that is also a fractional ideal. Elements of an order $\O$ have reduced norm and trace in $\ZZ$. An order is called \textit{maximal} when it is not contained in any other larger order. A suborder $\frakO$ of $\O$ is an order of rank 4 contained in $\O$. 

In this work, we will work with isomorphism classes of maximal orders in some quaternion algebra $\mathcal{B}$ and this is why we introduce the notion of type. 

\begin{definition}
\label{def: type}
The \textit{type} of an order $\O$ written $\Typ \O$ is the isomorphism class of $\O$.
\end{definition}

The left order of a fractional ideal is defined as $\O_L(I) = \{\alpha \in \mathcal{B}_{p,\infty}\;|\; \alpha I \subset I \}$ and similarly for the right order $\O_R(I)$. A fractional ideal is \textit{integral} if it is contained in its left order, or equivalently in its right order. An integral ideal is \textit{primitive} if it is not the scalar multiple of another integral ideal. We refer to integral primitive ideals hereafter as ideals. 

The product $IJ$ of ideals $I$ and $J$ satisfying $\O_R(I) = \O_L(J)$ is the ideal generated by the products of pairs in $I \times J$. It follows that $IJ$ is also an (integral) ideal and $\O_L(IJ) = \O_L(I)$ and $\O_R(IJ) = \O_R(J)$. The ideal norm is multiplicative with respect to ideal products. An ideal $I$ is invertible if there exists another ideal $I^{-1}$ verifying $I I^{-1} = \O_L(I) = \O_R(I^{-1})$ and $I^{-1} I = \O_R(I)=\O_L(I^{-1})$.
The conjugate of an ideal $\overline{I}$ is the set of conjugates of elements of $I$, which is an ideal satisfying $I\overline{I} = n(I)\O_L(I) $ and $\overline{I}I = n(I)\O_R(I)$.

We define an equivalence on orders by conjugacy and on left $\O$-ideals by right scalar multiplication.  Two orders $\O_1$ and $\O_2$ are equivalent if there is an element $\beta \in \mathcal{B}^\star$ such that $\beta \O_1  = \O_2 \beta$. Two left $\O$-ideals $I$ and $J$ are equivalent if there exists $\beta \in \mathcal{B}^\star$, such that $I = J \beta$. If the latter holds, then it follows that $\O_R(I)$ and $\O_R(J)$ are equivalent since $\beta \O_R(I) = \O_R(J) \beta$. For a given $\O$, this defines equivalences classes of left $\O$-ideals, and we denote the set of such classes by Cl$(\O)$. 


\paragraph*{The Deuring correspondence} is an equivalence of categories between isogenies of supersingular elliptic curves and the left ideals over maximal order $\O$ of $\mathcal{B}_{p,\infty}$, the unique quaternion algebra ramified at $p$ and $\infty$, inducing a bijection between conjugacy classes of supersingular $j$-invariants and maximal orders (up to equivalence) \cite{K96}. Moreover, this bijection is explicitly constructed as $E \rightarrow \End(E)$. Hence, given a supersingular curve $E_0$ with endomorphism ring $\O_0$, the pair $(E_1,\varphi)$, where $E_1$ is another supersingular elliptic curve and $\varphi : E_0 \rightarrow E_1$ is an isogeny, is sent to a left integral $\O_0$-ideal. The right order of this ideal is isomorphic to $\End(E_1)$. 
One way of realizing this correspondence is obtained through the kernel ideals defined in \cite{W69}. Given an integral left-$\O_0$-ideal I, we define the kernel of $I$ as the subgroup $$E_0[I]= \lbrace P \in E_0(\overline{\mathbb{F}}_{p^2}): \alpha(P)= 0 \mbox{ for all } \alpha \in I \rbrace.$$ To $I$, we associate the isogeny $$ \varphi_I : E_0 \rightarrow E_0/E_0[I].$$
Conversely, given an isogeny $\varphi$, the corresponding \textit{kernel ideal} is $$I_\varphi = \lbrace \alpha \in \O_0 \;:\; \alpha(P) = 0 \mbox{ for all } P \in \textnormal{ker}(\varphi) \rbrace.$$
In \cref{tab: deuring correspondence}, we recall the main features of the Deuring correspondence. 

\begin{table}[]
\centering
\begin{tabular}{ l @{\hspace{2em}} l }
\hline
  Supersingular $j$-invariants over $\mathbb{F}_{p^2}$ & Maximal orders in $\QA$  \\
  $j(E)$ (up to Galois conjugacy) & $\order \cong \textnormal{End}(E)$ (up to isomorpshim)\\
  \hline
  $(E_1,\varphi)$ with $\varphi: E \rightarrow E_1$ & $I_\varphi$ integral left $\order$-ideal and right $\O_1$-ideal\\
    \hline
  $\theta \in \End(E_0)$ & Principal ideal $\O \theta$ \\
  \hline
  deg$(\varphi)$ & $n(I_\varphi)$   \\
  \hline
  $\hat{\varphi}$ & $\overline{I_\varphi}$ \\
  \hline
  $\varphi : E\rightarrow E_1,\psi : E \rightarrow E_1$ & Equivalent Ideals $I_\varphi \sim I_\psi$ \\
  \hline
  Supersingular $j$-invariants over $\mathbb{F}_{p^2}$ & $\Cl(\O)$ \\
  \hline
  $\tau \circ \rho : E \rightarrow E_1 \rightarrow E_2 $ & $I _{\tau \circ \rho} = I_\rho \cdot I_\tau $ \\
  \hline
\end{tabular}
\caption{The Deuring correspondence, a summary from \cite{FKLPW20}. 
\label{tab: deuring correspondence}}
\end{table}

\paragraph*{Effective Deuring correspondence}

After establishing the nice theoretical results of the Deuring correspondence, it is natural to ask if we can obtain efficient algorithms to perform the translation between the two sides of our correspondence. This trend of work was started by Kohel, Lauter, Petit and Tignol in \cite{KLPT14}, and developed By Galbraith, Petit and Silva in \cite{GPS17}. 
In \cite{EHLMP18}, Eisentrager, Haller, Lauter, Petit and Morrison provided the first complete picture of the situation (at least heuristically). It turns out that if we start from the quaternion side (either as a maximal order or an ideal), there are polynomial-time algorithms to compute the corresponding element (j-invariant, or isogeny). In particular, Eisentrager et al. introduced a heuristic polynomial-time algorithm that computes the $j$-invariant corresponding to a maximal order type given in input. Henceforth, we call this algorithm \OrderToJ{}. It will be a crucial building block in one of our algorithm.

%% file: order_to_jinv.tex
Let us fix some prime number $p$. 

In this section, we introduce two algorithms to compute the $j$-invariants of supersingular curves over $\FF_{p^2}$ corresponding to a set $\frakS$ of maximal order types in $\QA$.
By the Deuring correspondence, we know that each maximal order type in $\QA$ corresponds to one or two $j$-invariants of supersingular curve over $\FF_{p^2}$.

We will explain in \cref{sec: hash} how to represent efficiently maximal order types as elements in some set $\cH$. Concretely, the input $\frakS$ to our algorithms will be given as some subset of $\cH$. 

Our two algorithms presented in \cref{sec: orders to jinv generic,sec: orders to jinv big set} target two opposite situations with respect to the relative size of $p$ and $\frakS$. The first algorithm is called \OrdersToJSmall{}. As the name suggests, it targets the case where $\# \frakS/p$ is "small". It is a quite direct application of standard results on the effective Deuring correspondence, and handles each maximal order independently.   
The asymptotic complexity is $O(\# \frakS \log p^{5+ \varepsilon})$ and works for any prime $p$ and set $\frakS$.   

The second algorithm is called \OrdersToJName{} and it is more involved in both design and analysis. It targets the case where $\frakS$ is made of a significant portion of all $O(p)$ possible types and is based on the idea that since $\frakS$ is big enough, the strategy that consists in going through the entire supersingular isogeny graph, collecting the $j$-invariants we want along the way, is quite optimal. 
Its complexity is $O(\frakS \log p^{2 + \varepsilon} + p \log p^{1+\varepsilon})$. Hence, the cutoff between the two methods will be for some $\frakS$ with $\# \frakS  =\Theta(p/ \log p^{2+ \varepsilon})$.
Note that our second algorithm  will be optimal when $p/\# \frakS = \Theta(\log p^{1 + \varepsilon})$. 

\subsection{Hashing to maximal order types}
\label{sec: hash}

One of the important point for making our algorithms practical is to have a good way to handle sets of maximal order types and test if a type belong in some set of types. 

For any maximal order $\O$, we will represent $\Typ \O$ by an invariant $H(\O)$. The purpose of this section is to introduce an efficiently computable invariant $H(\O)$ and the corresponding function $H$. 

To derive an invariant for an isomorphism classes of lattices, it is quite natural to look at the smallest elements of that lattice. This idea was introduced by Chevyrev and Galbraith \cite{chevyrev2014constructing} in a related context, and their result was recently strengthened by Goren and Love \cite{goren2023elements}. Let us take a maximal order $\O$. It can be shown (see \cite{chevyrev2014constructing} and \cite[Theorem 1.4]{goren2023elements}) that if $x_1,x_2,x_3$ realize the successive minimas of the Gross lattice $\O^T = \lbrace 2 x - \tr(x) |  x \in \O \rbrace$, then the values $n_1,n_2,n_3$ where $n_i = \norm{x_i}$ for $i=1,2,3$ determine uniquely the maximal order $\O$. 

This is enough to obtain an invariant of size $O(\log p)$ as it can be shown that $\log (n_i) = O(\log p)$ for all $i \in \lbrace 1,2,3 \rbrace$. 
If needed, for compactness, one can then apply some kind of hash function $h : \lbrace 0,1 \rbrace^{*} \rightarrow \cH$ where $\cH$ is big enough to make the probability of a collision negligible over all maximal order types in $\QA$. 

For a generic statement, we take an arbitrary function $h$ (which might be the identity) and assume its computational cost is negligible. Then, we define $H$ as:

\begin{enumerate}
    \item Compute $\O^T$.
    \item Compute the three successive minimas $x_1,x_2,x_3$ of $\O^T$.
    \item Output $H(\O) = h\left(\norm{x_1},\norm{x_2},\norm{x_3}\right)$.  
\end{enumerate}

\begin{proposition}
\label{prop: hash complexity}
The hash function $H$ presented above can be computed in $O(\log(X)^{1+\varepsilon})$ when all the coefficients in the decomposition of the basis of $\O$ over $\G{1,i,j,k}$ are smaller than $X$. 
\end{proposition}

\begin{proof}
This can be achieved by computing the successive minimas of $\O^T$ with the algorithm to reduce ternary quadratic form of \cite{eisenbrand2001fast}, or the $\Tilde{L}^1$ algorithm from \cite{novocin2011lll} to perform lattice reduction in small dimension. 
\end{proof}

\subsection{Matching maximal orders and $j$-invariant: a generic algorithm.}
\label{sec: orders to jinv generic}

Let us fix a prime $p$. 
We assume that a function $H$ as introduced in \cref{sec: hash} is defined and we assume that the underlying hash function $h$ is such that there is no collisions over all maximal order types in $\QA$.   

The algorithm \OrdersToJSmall{} is easily derived from the ideal-to-isogeny algorithm described as \cite[Algorithm 12]{EHLMP18} and whose main building blocks where first introduced in \cite{GPS17}. We provide a somewhat detailed description of this algorithm below. We refer to \cite{GPS17,eriksen2023deuring} for a description of the main algorithmic building blocks. 

\begin{enumerate}
    \item Compute a curve $E_0$ and a maximal order $\O_0$ with $\End(E_0) \cong \O_0$.
    \item Set $M = \emptyset$.   
    \item For each $\O \in \frakS$ do:
    \item \qquad Compute an ideal $I$ of powersmooth norm connecting $\O_0$ and $\O$
    \item \qquad Translate that ideal into the corresponding isogeny $\varphi_I$. 
    \item \qquad Compute the $j$-invariant $j_\O$ of the codomain of $\varphi_I$.
    \item \qquad $M = M \cup \lbrace H(\O), j_\O \rbrace$.
    \item Return $M$.   
\end{enumerate}

Under some heuristic assumptions regarding the distribution of numbers represented by quadratic forms (see \cite{KLPT14}), the complexity of this algorithm is $O( \# \frakS \log p^{5+\varepsilon})$.

For a concrete implementation, we propose to use the approach described in \cite{eriksen2023deuring}.

\subsection{Maximal orders to $j$-inv: a faster algorithm when $\# \frakS \approx p$. }
\label{sec: orders to jinv big set}

 \OrdersToJSmall{} treats independently each order in $\frakS$. This is fine when $\frakS$ is quite small compared to the number of maximal order types in $\QA$, but it becomes less and less optimal as $\#\frakS/p$ increases. Indeed, we end up redoing a lot of computations in that manner.  

To overcome this problem when the ratio $\#\frakS/p$ increases, we can try to mutualize as much computation as possible by using an approach that explores the entire isogeny graph and only targets the specific elements of $\frakS$ along the way. Since, our exploration of the isogeny graph is completely generic, we obtain an algorithm with a much better amortized cost per maximal order type.    

More concretely, our idea is the following: take a smooth degree $L$ such that all supersingular curves are $L$-isogenous to some starting curve $E_0$ of endomorphism ring in $\Typ \O_0$ for some maximal order $\O_0$. Compute all the $\O_0$-ideals of norm $L$ and their right orders and select the ones contained in the set $\frakS$. Then, enumerate efficiently through all the corresponding isogenies of degree $L$ and collect the $j$-invariants of the codomains. 
Since quaternion operations cost less, we will do the exhaustive part over the quaternions, while minimizing the cost of elliptic curve operations by "selecting" the $L$-isogenies than we cannot avoid to compute. 

Note that we do the ideal and isogeny phases in a simultaneous manner in \OrdersToJName{}. The good complexity we obtain will come from the care we take in computing all the required isogenies in the most efficient way possible, and in particular to avoid as many useless isogeny computations as possible.  
For that, the choice of $L$ will be very important. In particular, if $L = L_1 L_2$, by factorization of isogenies, we can compute all relevant $L$-isogenies by computing all $L_1$-isogenies and only a subset of all the $O(L_1 L_2)$ $L_2$-isogenies. This subset obviously depends on $\frakS$ and the $L_1 L_2$-isogeny computations account for the 
$\frakS \log p^2$ terms in the complexity (because we will choose $L_2$ to be smooth). 
The $p \log p$ covers the costs of the quaternion operations (which does not depend on $\frakS$ since we cover all maximal order types) and $L_1$ isogeny computations. 

Before describing the algorithm in itself, we need to provide several properties and one heuristic claim that are going to be crucial for analyzing the algorithm complexity and proving its correctness.

\paragraph*{Preliminary results and a heuristic assumption.} Our first results target the degree $L$ of the isogenies that we will use. As we explained above, this degree is crucial for optimizing the algorithm. To enable the fast computation of a lot of $L$-isogenies at the same time, we need it to be power-smooth (for efficient application of the Vélu formulas), but with coprime factors that are not too small.   

Let us write $\Phi(N)$ the number of cyclic isogenies of degree $N$ for any $N \in \NN$. 

\begin{lemma}
\label{claim: log factor basis}
There exists a bound $B$ and constants $B_0,C_0,  1 < C_1 < C_2$ such that for every number $N > B$, there exists a value $n < B_0 \log(N)$ and integers $(e_i)_{1 \leq i \leq n}$
 such that, if we define $\ell_i$ to be the $(n-i+1)$-th smallest prime and $L_i = \ell_i^{e_i}$ then we have $\sqrt{C_0 \log(N)} \leq \Phi(L_i) < C_0 \log(N)$ for all $1 \leq i \leq n-1$, $(C_0/2) \log(N) \leq \Phi(L_n) < C_0 \log(N) $ and $C_1 N \leq \prod_{i=1}^n \Phi(L_i) < C_2 N$. 
\end{lemma}

\begin{proof}

    We remind that when $\ell$ is a prime number and $e \geq 1$ is an integer, $\Phi(\ell^e) = \ell^{e-1}(1 + \ell)$.
    
    Let us take $C_0, 1 <  C_1 < C_2$ three constants. We leave these constants unspecified for now. We start our reasoning without assuming anything on those constants, and we will encounter conditions on them during the proof. At the end, we will verify that these constraints can all be satisfied.
    
    We also take a bound $B$ that we assume to be "big enough" for various asymptotic inequalities to be verified. 
    Let us finally take any integer $N> B$.

    Let us construct two integers $n,\Lambda_n$ recursively from $\Lambda_0 = 1$ in the following manner: if $\Lambda_i \geq C_1 N$, then set $n=i$, otherwise let $\Lambda_{i+1} = \Phi(\lambda_{i+1}) \Lambda_i$ with $\lambda_{i+1} =\ell_{n-i}^{e_{n-i}}$ where each $\ell_{i}$ is the $n-i+1$-th prime and $e_{i}$ is defined to be $1$ if $\Phi(\ell_i) \geq C_0 \log(N)$ or the biggest exponent such that $\Phi( \lambda_{i+1}) < C_0 \log(N)$.

    It is clear that this recursive algorithm always terminates and so we can get two integers $n,\Lambda_n$ in that manner for any $N$. Moreover, we have $C_0/2 \log(N) \leq \Phi(\lambda_1)$ since $\ell_1 =2$ and $\sqrt{C_0 \log(N)} \leq \Phi(\lambda_i)$ for all other $i$ (because if $\ell_i \leq \Phi(\lambda_i) < \sqrt{C_0 \log(N)}$, then $\ell_{n-i} \Phi(\lambda_i) = \Phi( \ell_{n-i} \lambda_i ) \leq C_0 \log(N)$ and this is impossible by construction of the exponent $e_{n-i}$).

   Then, there are two possibilities: either $\Lambda_n < C_2 N$ or $\Lambda_n \geq C_2 N$. 

   \paragraph*{First case: $\Lambda_n < C_2 N$.} Let us set $L_i = \lambda_{n-i +1}$ for $1 \leq i \leq n$, we will show that $L_1,\ldots,L_n$ is satisfies all the required properties. 

   In fact, the only thing that remains to be verified is that $\Phi(L_i) < C_0 \log(N)$ for $1 \leq i \leq n-1$. 
   By construction of the $L_i$, it is clear that is suffices to prove that the $n$-th  prime (the divisor of $L_1$) is smaller than $C_0 \log(N)-1$. 
   
   By taking logarithms of the inequality $ \sqrt{C_0 \log(N)}^n \leq \Lambda_n < C_2 N$, 
   we otbain $ (n/2) ( \log(C_0) + \log \log(N)) < \Lambda_n < \log (C_2) + \log (N)$. From there, it follows (since $B$ is big enough) that $n < 3 \log(N) / \log \log(N)$.

   From asymptotic results on the size of the $n$-th prime (see for instance \cite[3.13]{RS62}), we get that we can assume $\ell_1 < 2 n \log (n) < 6 \log(N)$. 
   Thus, if $C_0 \geq 7$, the result is proven.

   \paragraph*{Second case: $\Lambda_n \geq C_2 N$. } Since $\Lambda_n$ is too big, we will try to divide $\Lambda_n$ by a factor $\delta$. Since $N$ is big enough, we can assume $n \geq 4$. We don't want to modify $\lambda_1$ since the bound is tight, but we can divide $\lambda_2,\lambda_3$ by some powers of $\ell_{n-1} =3,\ell_{n-2} =5$. The idea is that we can remove up to $O(\sqrt{\log(N)})$ from $\lambda_2,\lambda_3$.
   

    Since we have $\Lambda_{n-1} < C_1 N$, by the same reasoning as before we get $n < 1 +  3 \log(N) / \log \log(N)$. Hence, with $\ell_1 < 2 n \log(n) $ we get $\ell_1 < 7 \log(N)$.

    Since, $\Phi(\lambda_n)$ is either smaller than $C_0 \log(N)$ or equal to $\ell_1+1$, we deduce that $\Phi(\lambda_n) < 
    \max(8,C_0) \log(N)$ and it suffices to choose $C_0 \geq 8$ to ensure that $\Phi(\lambda_n) < C_0 \log(N)$. 

    From there, we deduce 
    $\Lambda_n  = \Lambda_{n-1} \Phi(\lambda_n) < C_0 C_1 N \log(N)$, from which we get 
    \begin{equation}
        \label{eq: bound 1}
         C_1 N \leq \Lambda_n < C_0 C_1  N\log(N).
    \end{equation} 
    Now, we want to find $\delta | \Lambda_n $ such that $C_1 N \leq \Lambda_n / \delta < C_2 N$. We start by using $\lambda_2$. Let us take $\delta_2$ as the biggest divisor of $\Phi(\lambda_2)$ such that $$\delta_2 < \min ( \Lambda_n / (C_1 N), \sqrt{C_0 \log(N)}/ 3 ).$$  

    If we assume $C_2 > 3 C_1$, then we can clearly take $\delta_2  \geq 3$ without having $\Lambda_n/\delta_2 \leq C_1 N$. 
    
    If $\Lambda_n /\delta_2 < C_2 N$, then we are done with $L_{n-1} = \lambda_2/\delta_2$ and $L_i = \lambda_{n-i+1}$ for all other $1 \leq i \leq n$, as we have $ C_1 N \leq \Lambda_n /\delta_2 < C_2 N$ and the bound $\Phi(\lambda_2) / \delta_2 > \sqrt{C_0 \log(N)}$ is satisfied by choice of $\delta_2$ and the fact that $\Phi(\lambda_2) > C_0 \log(N) /3$.


    If not and $\Lambda_n /\delta_2 \geq C_2 N$ , we will now try to remove a divisor of $\lambda_3$. Before that, it is useful to try to upper-bound $\Lambda_n /\delta_2$ more precisely. 

    We have $\delta_2 > \min ( \Lambda_n / (C_1 N), \sqrt{C_0 \log(N)}/3 ) /3$ (otherwise we could have multiplied $\delta_2$ by $3$).
    If the inequality $\Lambda_n / (C_1 N) < \sqrt{C_0 \log(N)}/3 $ were to be true, then we would have $\delta_2 > \Lambda_n / (3 C_1 N) $
    and so $\Lambda_n / \delta_2 < 3 C_1 N$. This is smaller than $C_2 N$ by our assumption on $C_1$ and $C_2$, and so this is a contradiction. Thus we must have $\Lambda_n / (C_1 N) \geq \sqrt{C_0 \log(N)}/3$ and we can get the lower-bound $\delta_2 \leq \sqrt{C_0 \log(N) }/9$. If we plug that into \cref{eq: bound 1}, we obtain the following  
        \begin{equation}
        \label{eq: bound 2}
         C_1 N \leq \Lambda / \delta_2 < 9 \sqrt{C_0} C_1 N \sqrt{\log(N)}.
    \end{equation} 

    Now, we define $\delta_3$ as the biggest divisor of $\Phi(\lambda_3)$ such that 
    $$\delta_3 < \min ( \Lambda_n / (C_1 \delta_2 N), \sqrt{C_0 \log(N)}/\ell_3 ).$$ 
    Then, following the same reasoning as for $\delta_2$, we get that if we assume further $5 C_1 < C_2$, then we can take a non-trivial $\delta_3$. 
    Furthermore, we have that either $\Lambda_n/(\delta_2 \delta_3) < C_2 N$ and we are done (for the same reasons as before), or we must have  $\Lambda_n / (\delta_2 C_1 N) \geq \sqrt{C_0 \log(N)}/\ell_3$. 
    
    In that case, we can plug this in \cref{eq: bound 2} to get 
         \begin{equation}
        \label{eq: bound 3}
         C_1 N \leq \frac{\Lambda_n}{\delta_2 \delta_3}  <  225 C_1 N.
    \end{equation} 
    If we assume that $C_2 > 225 C_1$ this is impossible and so we are done. 

    In summary, suitable integers $L_1,\ldots,L_n$ can be found for any big enough $N$ assuming that the constants $C_0,C_1,C_2$ satisfies:
    \begin{itemize}
        \item $C_0 > 7$.
        \item $C_2 > 3 C_1$.
        \item $C_2 > 5 C_1$.
        \item $C_2 > 225 C_1$.
    \end{itemize}
    This system of inequalities can clearly be satisfied and so this proves the result.  
    
    \end{proof}

\begin{definition}
    \label{def: good degree basis}
    Given any integer $N>B$, the integers $L_1,\ldots,L_n$ as defined in \cref{claim: log factor basis} is called a \textbf{good-degree basis} for $N$. 
\end{definition}

We derive the following result that will be useful in our analysis.

\begin{proposition}
\label{prop: log factor basis critical bound}
Take an integer $N> B$ with good degree-basis $L_1,L_2,\cdots,L_n$ with $n > 7$. Then, there exists a constant $C_3$ such that 
\begin{equation}
    \label{eq: log factor basis critical bound}
    \sum_{i=1}^{n-6} (n-i)\prod_{j=1}^i \Phi(L_j) < C_3 \frac{N}{\log(N)^2}
\end{equation}
\end{proposition}
\begin{proof}

To prove the result, we are going to use the equalities $\sum_{i=0}^{m} X^{i} = \frac{X^{m+1} -1}{X-1}$ and $\sum_{i=1}^{m} i X^{i} = X ( \frac{mX^{m+1}-(m+1)X^m +1}{(X-1)^2} )$ for any $m$ and $X$. 
By our definition of a degree-basis in \cref{claim: log factor basis}, we have that $\Phi(L_j) \geq \sqrt{C_0 \log(N)}$ for all $1 \leq j \leq n$.
Thus, $\prod_{j=1}^i \Phi(L_j) < C_2 N / \prod_{j=i+1}^n \Phi(L_j) < C_2 N/ (\sqrt{C_0 \log(N)})^{n-i}$. Let us take $X = 1/\sqrt{C_0 \log(N)}$. 
We have $\sum_{i=1}^{n-6} (n-i) \prod_{j=1}^i \Phi(L_i) < C_2 N \sum_{i=1}^{n-6} (n-i) X^{n-i}$.
Then, we have

$\sum_{i=1}^{n-6} (n-i) X^{n-i} = \sum_{i=6}^{n-1} i X^i = X^6 \sum_{i=0}^{n-7} (i+6)X^i$. Then, we apply our two equalities to get

\begin{align*}
    \sum_{i=1}^{n-6} (n-i) X^{n-i} = & X^6 (6 \frac{X^{n-6}-1}{X-1}) + \frac{(n-7)X^{n-5} - (n-6)X^{n-6} +X }{(X-1)^2}) \\
    = & X^6 \frac{(n-1)X^{n-5} - (n-12)X^{n-6} - 5X + 6}{(X-1)^2}.
\end{align*}

Since, asymptotically, we will have $X < 1$ while $n$ will increase, we have that the leading term in the numerator of the fraction is 1, and so there exists $C'_3$ such that $\sum_{i=1}^{n-6} (n-i) X^{n-i} \leq C_3' X^4$. Thus, $\sum_{i=1}^{n-6} (n-i) \prod_{j=1}^i \Phi(L_i) < \frac{C_2 C_3'}{C_0^2} \frac{N}{\log(N)^2}$. 

\qed
\end{proof}

Our heuristic claim is about the size of $L$ required to meet the condition that all supersingular curves are $L$-isogenous to some curve $E_0$. We rewrite this under the Deuring correspondence as a condition on maximal orders and ideals. 

\begin{claime}
\label{claim: number of maximal orders} 
There exists a constant $C_4$ such that for any prime $p$ and maximal order $\O_0$ in $\QA$, given any number $N > pC_4$, every maximal order type in $\QA$ is obtained as the type of the right order of a left integral $\O_0$-ideal of norm $N$.    
\end{claime}


\begin{remark}
\cref{claim: number of maximal orders} is consistent with experiments regarding the diameter of the graph of supersingular $2$-isogenies made in \cite{arpin2021adventures}. 
We also made some small experiments that seems to be consistent with that idea. 

In case our claim fails, it is possible to use several starting curve $E_0$ to decrease the probability of missing some curve. 
If this is still not enough, and a few types are not obtained in this manner, it is always possible to apply $\OrderToJName{}$ to compute the remaining $j$-invariants without damaging too much the complexity.    
\end{remark}

For a given input $p$ to $\OrdersToJName{}$, we assume the knowledge of a good degree-basis $L_1,\cdots,L_n$ (see \cref{def: good degree basis}) of $C_4p$ (where $C_4$ is the constant in \cref{claim: number of maximal orders}).
In \OrdersToJName{}, we will use the $L_i$ torsion points for all $i$. For supersingular curves, these points can always be defined over an extension $\FF_{p^{m_i}}$. We have the following lemma to bound the value of all the $m_i$ from $L_i$. 

\begin{lemma}
\label{lemma: torsion degree}
Let $p$ be a prime number and $E_0$ a supersingular curve over $\FF_{p^2}$. For any integer $N$, coprime with $p$, there exists $\hat{E_0} \cong E_0$ over $\overline{\FF}_{p}$ such that the torsion subgroup $\tilde{E_0}[N]$ is defined over an extension $\FF_{p^{m}}$ of even degree $m \leq N$ over $\FF_{p}$. 
\end{lemma}

\begin{algorithm}[H]
 \caption{$\OrdersToJName{}$}\label{alg: orders to j-inv}
 \begin{algorithmic}[1]
 \Require {A prime $p$. , and a set of maximal order types in $\QA$. }
 \Ensure {The set of $j$-invariants corresponding to the maximal orders of $\frakS$.}
 \STATE \label{step: good degree basis} Establish $L_1,\ldots,L_n$, a good-degree basis for the number of supersingular curves. 
\STATE\label{step: compute starting curve} Compute a supersingular curve $E_0$ over $\FF_{p^2}$ with known endomorphism ring $\O_0$. 
\STATE\label{step: compute starting ideal} Compute $I_0 = \O_0 \G{\alpha_0,L}$ one left integral $\O_0$-ideal of norm $L := \prod_{i=1}^n L_i$. 
\STATE\label{step: compute rotating endomorphism} Find $\alpha \in \End(E_0)$ such that $\gcd( \norm{x + y\alpha},L) =1$ for all $x,y$ with $\gcd(x,y,L) =1$.
\FOR{$i \in [1,\ldots,n]$}\label{loop: basis computation}
    \STATE Compute a basis $P_{0,i},Q_{0,i}$ of $E_0[L_i]$ over $\FF_{p^{m_i}}$. 
    \STATE Compute the generator $R_{0,i}$ of $E[I_0 + \O_0 L_i]$ from $P_i,Q_i$. 
    \STATE Compute $S_{0,i}=\alpha(R_{0,i})$.
\ENDFOR
\STATE Set $\texttt{List} = [\lbrace E_0, \O_0 \G{1} , [R_{0,1},\ldots,R_{0,n}],[S_{0,1},\ldots,S_{0,n}] \rbrace ]$.
 \STATE Set $M = \emptyset$ and $m = 0$. 
\IF {$H(\O_0) \in \frakS$} 
    \STATE $M = M \cup \lbrace (H(\O_0),j(E_0)) \rbrace$ and $m=m+1$.
\ENDIF

\FOR{$i \in [1,\ldots,n]$}\label{loop: cofactors}
    \STATE \texttt{NewList}= $[]$. 
    \FOR {$x \in \texttt{List}$}\label{loop: inner list}
        \STATE Parse $x$ as $E,I,[R_i,\ldots,R_n],[S_i,\ldots,S_n]$. 
        \FOR{ all cyclic subgroups $\G{C R_i + [D]S_i}$ of order $L_i$ in $\G{R_i,S_i}$ }\label{loop: isogenies}
            \STATE Compute $P := [C] R_i + [D] S_i$. 
            \STATE Compute $J$ the ideal $\O_0 \G{ \alpha_0 (C + D \overline{\alpha}),L_i }$. 
            \STATE Set $K := I \cap J$.
            \STATE Set $\texttt{List}_1 = []$, $\texttt{List}_2 = []$.  
            \STATE Let $\O = \O_R(K)$. 
            \IF{ $i<n$ }
                \STATE\label{step: isogeny computation} Compute $\vphi: E \rightarrow E/\G{P}$.
                \STATE\label{step: evaluation1} Compute $\texttt{List}_1 = [\vphi(R_{i+1}),\ldots,\vphi(R_{n})]$.
                \STATE\label{step: evaluation2} Compute $\texttt{List}_2 = [\vphi(S_{i+1}),\ldots,\vphi(S_{n})]$.
            \ENDIF 
            \IF {$H(\O) \in \frakS$ and not contained in $M$ already} 
                \IF{ $i=n$ }
                    \STATE\label{step: final isogeny computation} Compute $\vphi: E \rightarrow E/\G{P}$.
                \ENDIF 
                \STATE $M = M \cup \lbrace (H(\O),j(E/\G{P}) ) \rbrace$ and $m=m+1$.
            \ENDIF
            \IF{$m = \# \frakS$} 
                \STATE Return $M$.
            \ENDIF 
            \STATE Concatenate \texttt{NewList} and $[ E/\G{P},K,\texttt{List}_1,\texttt{List}_2]$.
        \ENDFOR
    \ENDFOR
    \STATE \texttt{List} = \texttt{NewList}.
\ENDFOR
\RETURN $M$. 
\end{algorithmic}
\end{algorithm}

\begin{proposition}
\label{prop: correct orders to j-inv}
Assuming \cref{claim: number of maximal orders}, \OrdersToJBig{p} is correct. 
\end{proposition}
\begin{proof}
By \cref{claim: number of maximal orders} and our choice of $L$, we see that each maximal order types in $\QA$ is obtained as the right order of a $L$-ideal. Thus, we need to prove that our algorithm goes through every possible $L$-ideal and that it computes correctly the $j$-invariant associated with the right orders of those ideals. 
We will make our reasoning over isogenies and the Deuring correspondence will allow us to conclude the result over ideals.  

Every cyclic $L$-isogeny can be factored as $\vphi_n \circ \ldots \circ \vphi_1$ where $\vphi_i$ is an isogeny of degree $L_i$. 
When $R_i,S_i$ is a basis of $E[L_i]$, then all cyclic subgroups of order $L_i$ are generated by an element $[C] R_i + [D] S_i$. There is a $1$-to-$1$ correspondence between cyclic subgroups of order $L_i$ and cyclic isogenies of degree $L_i$. 
Since $R_{0,i},S_{0,i}$ is a basis of $E_0[L_i]$ and $\vphi_{i-1} \circ \cdots \circ \vphi_1$ has degree coprime with $L_i$, the two points $R_i,S_i$ are a basis of $E[L_i]$ and so this proves that our enumeration covers all possible isogenies of degree $L_i$ at each iteration of index $i$ of the loop in line~\ref{loop: isogenies}. 

Thus, at the end of the loop we have covered all $L$-isogenies, and this means that if our ideal computation is correct, then we have covered all maximal order types and our algorithm is correct.   

Remains to prove that the ideal $I_1 \cap \ldots \cap I_i$ where each $I_j = \O_0 \G{\alpha_0 (C_j + D_j \overline{\alpha}),L_j}$ is the ideal corresponding to the isogeny $\vphi_i \circ \ldots \vphi_1$ where each $\vphi_j$ has kernel $ \vphi_{j-1} \circ \cdots \circ \vphi_1 ( [C_j] R_{0,j} + [D_j] S_{0,j})$ or equivalently that the kernel of $\vphi_i \circ \ldots \circ \vphi_i = \sum_{j=1}^i ([C_j] R_{0,j} + [D_j] S_{0,j})$. 

For that, it suffices to prove the result for each coprime factor of the degree, so we need to prove that $E_0[I_j] = \G{[C_j] R_{0,j} + [D_j] S_{0,j}}$. Let us go back to the definition of an ideal kernel given in \cref{sec: background}. 
We have $E_0[I_j] = \lbrace P, \beta (I_J) = 0  \forall \beta \in I_j \lbrace$. Since $I_j$ contains $L_j \O_0$, it is clear that the kernel must be a subgroup of $E_0[L_j]$. Since multiplication in $\QA$ amounts to composition of the corresponding isogenies, it suffices to verify that $\ker \alpha_0 (C_j + D_j \overline{\alpha}) \cap E_0[L_j] = \G{[C_j] R_{0,j} + [D_j] S_{0,j}} $. 

First, note that we have $\ker \alpha_0 \cap E_0[J_j] = \G{R_{0,j}}$ by definition of $\alpha_0$ and $R_{0,j}$. 
Then, we have $([C_j] + [D_j] \overline{\alpha}) ([C_j] R_{0,j} + [D_j] S_{0,j}) = ([C_j] + [D_j] (\overline{\alpha}) ([C_j] + [D_j] \alpha) (R_{0,j})  = [\norm{C_j + D_j \alpha}] R_{0,j}$.
This proves that we have $\G{[C_j] R_{0,j} + [D_j] S_{0,j}} \subset \ker \alpha_0 (C_j + D_j \overline{\alpha}) \cap E_0[L_j]$. 

By definition of $\alpha$ the scalar $\norm{C_j + D_j \alpha}$ is coprime with $L_j$ and so the endomorphism $C_j + D_j \alpha$ is a bijection on $E_0[L_j]$. Thus, there cannot be another subgroup than $\G{[C_j] R_{0,j} + [D_j] S_{0,j}}$ that is sent to $\G{R_{0,j}}$ and this concludes the proof that $\ker \alpha_0 (C_j + D_j \overline{\alpha}) \cap E_0[L_j] = \G{[C_j] R_{0,j} + [D_j] S_{0,j}} $.

With that last fact, we have proven the point. 
\end{proof}

\subsubsection*{Complexity analysis. }

Below, we give as \cref{thm: complexity}, a complexity statement for \cref{alg: orders to j-inv}. We derive this result from smaller statements for all the main Steps of \cref{alg: orders to j-inv}. The proofs of these statements include a more detailed description of the steps when needed. When a step has a complexity that will end-up being negligible before the total cost, we will sometimes not bother with a precise statement.  

Note that all operations involving manipulations of the list $M$ can be done efficiently in $O(1)$ using adequate data structure so we do not analyze this part of the computation. 

Since we look for an asymptotic statement, we assume that the size $n$ of the good-degree basis used is bigger than $6$ so we can apply \cref{prop: log factor basis critical bound}.

\begin{proposition}
\label{prop: complexity starting elements}
Under GRH, Step~\ref{step: compute starting curve} can be executed $\poly{\log p}$. 
\end{proposition}
\begin{proof}
The first step can be performed using an algorithm that was described as part of the proof of \cite[Proposition 3]{EHLMP18}. The idea is the following. Select the smallest fundamental discriminant $d$ such that $p$ is inert in the ring of integer $R_d$ of $\QQ(\sqrt{d})$. Note that it can be proven that $d = O(\log p^2)$ under GRH. Then, we know that there are supersigular curves that admit an embedding of $R_d$ in their endomorphism ring. The $j$-invariants of these curves are the roots to the Hilbert class polynomial $H_d$. 
It suffices to find one root of $H_d$ over $\FF_p$ to get a supersingular curve $E_0$ whose endomorphism ring will contain the Frobenius $\pi$ and an endomorphism $\iota$ of norm $d$. This endomorphism can be found by computing all the isogenies of degree $d$ with the Vélu formulae. Since the suborder $\G{1,\iota,\pi,\iota \circ \pi}$ has an index in $O(d) = \poly{\log p}$ inside $\End(E_0)$, and recovering the full endomorphism ring can be done in $\poly{\log p}$. 

This proves the result for Step~\ref{step: compute starting curve}.

\end{proof}

\begin{proposition}
\label{prop: complexity rotating endomorphism}
 Step~\ref{step: compute starting ideal} and Step~\ref{step: compute rotating endomorphism} can be executed in $\polylog{p}$ and the output $\alpha_0$ and $\alpha$ have coefficients in $\poly{p}$ over the canonical basis of $\QA$. 
\end{proposition}
\begin{proof}
First, note that we can fix a basis of $\O_0$ with coefficients (over the canonical basis of $\QA$) in $O(p)$ (see \cite[Section 2.3]{KLPT14} for instance). 

The Steps~\ref{step: compute starting ideal} and \ref{step: compute rotating endomorphism} can be solved in a similar manner despite a final goal that is quite different. 
For Step~\ref{step: compute starting ideal}, it is sufficient to find a $\beta_0$ such that that the matrix of the action of $\beta_0$ on a basis of the $L$-torsion has two distinct eigenvalues. Then, we can take $\alpha_0 = \beta_0 - \lambda$ where $\lambda$ is one of the two eigenvalues. Note that this is equivalent to saying that $\beta_0$ needs to have two distinct eigenvalues mod $L_i$ for all $1 \leq i \leq n$.  

For Step~\ref{step: compute rotating endomorphism}, a sufficient condition to obtain $\alpha$ such that $\gcd( \norm{x + y\alpha},L) =1$ for all $x,y$ with $\gcd(x,y,L) =1$ is to have that the matrix of $\alpha$ over a basis of the $L_i$-torsion with no eigenvalues for all $1 \leq i \leq n$. 

The existence and value of these eigenvalues $\mod L_i$ for a quaternion element $\beta$ can be verified directly by computing the roots of the polynomial $X^2 + \tr(\beta) X + \norm{\beta} \mod L_i$.  

Thus, to solve the two steps, we can apply the following method. 
First, for each $L_i$ dividing $L$, find one element $\beta_{0,i}$ (resp. $\alpha_i$) in $\O_0 /L_i \O_0 $ with two distinct (resp. no) eigenvalues $\mod L_i$. 
Since the ring $\O_0 / L_i \O_0$ is isomorphic to $\MM_2(\ZZ/L_i \ZZ)$ it is clear that a solution can be found by enumerating over the $L_i^4$ elements of $\O_0 / L_i \O_0$. 

Then, the final element $\beta_0$ (resp. $\alpha$) can be obtained by CRT (doing this coefficient-wise over the basis of $\O_0$). The coefficients of $\beta_0$ (resp. $\alpha$) will have size $O(L)$ over the basis of $\O_0$ and so we get the desired result by taking into account the coefficients of this basis over the canonical basis of $\QA$.
For each $L_i$, we have to enumerate at most $L_i^4$ quaternion elements, so the final complexity statements follows from the complexity of computing modular squareroot and the complexity of CRT.
\end{proof}

\begin{proposition}
\label{prop: complexity loop basis}
The FOR loop in line~\ref{loop: basis computation} can be executed in $\polylog{p}$. 
\end{proposition}
\begin{proof}
The loop is repeated $O(\log p)$ times. 
Computing a basis is a very standard task and it can be done in $\polylog{p}$ since $m_i = O(\log p)$. 
The generator $R_{0,i}$ can be computed in $\polylog{p}$ using the algorithm described in \cite{GPS17} to find the kernel of an ideal. 
Evaluating the endomorphism $\alpha$ can be done by evaluating a basis of $\End(E_0)$ and then performing the scalar multiplications corresponding to the coefficients of $\alpha$ in this basis. The first part can be done in $\polylog{p}$ with the ideas of the proof of \cref{prop: complexity starting elements} by choice of $E_0$ and the second can also be done in $\polylog{p}$ by the size bound given on the coefficients of $\alpha$ in \cref{prop: complexity rotating endomorphism}. 
\end{proof}

The main computational task of \OrdersToJName{} is quite clearly performed during the loop in line~\ref{loop: cofactors}. This FOR loop contains two inner loops. We will incrementally provide complexity statements for each of these loops in orders to clearly decompose the cost of each operations.

\begin{proposition}
\label{prop: complexity and correction step isogeny computation} 
At index $i<n - 6$,
Step~\ref{step: isogeny computation} produces a polynomial over $\FF_{p^2}$ of degree smaller than $L_i$ in $O(M_{\PP}(L_i)\log(L_i))$ operations over $\FF_{p^{m_i}}$. This polynomial defines uniquely the isogeny $\vphi$. 
At index $n-6 \leq i \leq n$, $L_i = \ell_i^{e_i}$ for some prime $\ell_i \leq 13$ (the $6$-th prime) and Step~\ref{step: isogeny computation} produces $e_i$ polynomials of degree $\ell_i$ over $\FF_{p^2}$ in $O(e_i^2)$ operations over $\FF_{p^{m_i}}$. These polynomials uniquely defines the isogeny $\vphi$.
\end{proposition}
\begin{proof}

It can be shown with the Vélu formulaes \cite{V71}, that to represent an isogeny it suffices to compute its kernel polynomial,  \ie the polynomial whose roots are the x-coordinate of the points of the kernel and that this polynomial is always defined over $\FF_{p^2}$ even if the kernel points are not. 

From a single kernel point, the $O(L_i)$ points of the kernel can be generated in $O(L_i) $ operations over $\FF_{p^{m_i}}$. Then, the kernel polynomial can be constructed with complexity  $O(M_{\PP}(L_i)\log(L_i)) $ from its roots. 

When $n-6 \leq i \leq n$, we can write $L_i = \ell_i^{e_i}$ for some prime $\ell_i =O(1)$ and then, we can factor our isogeny of degree $L_i$ as $e_i$ isogenies of degree $\ell_i$. All these isogenies can be computed in time $O(e_i^2)$ from a kernel generator (see \cite{DFJ11} for more on this topic). 
\end{proof}

\begin{proposition}
\label{prop: complexity evaluation step loop isogenies}
There exists a constant $C$ such that, at any index $i\leq n-6$, the number of $\FF_p$-operations executed in Steps~\ref{step: evaluation1},\ref{step: evaluation2} of the FOR loop in line~\ref{loop: isogenies} is upper bounded by $C (n-i) \log p  M_{\PP} (\log p) $.
When $i > n-6$, we have the upper-bound $C (n-i) \llog(p)^2  M_{\PP} (\log p) $
\end{proposition}
\begin{proof}
For each $i<j\leq n$,
we need to evaluate the polynomial produced by Step~\ref{step: isogeny computation} on the points $R_{i+1},\ldots,R_n$ and $S_{i+1,\ldots R_n}$. 
By \cref{prop: complexity and correction step isogeny computation}, when $i \leq n-6$, each evaluation costs $O(L_i)$ operations over $\FF_{p^{m_j}}$. The cost of arithmetic over $\FF_{p^{m_j}}$ in operations over $\FF_p$ is upper-bound by $C' M_{\PP} (m_j)$ for some constant $C'$.
So we get the result from $L_i =O(\log p)$ and $m_j =O (\log p)$ by \cref{lemma: torsion degree}. 

When $i > n-6$, there are $e_i = O (\llog(p))$ polynomials of degree $\ell_i = O(1)$ and we get the desired result. 

\end{proof}

\begin{remark}
The Vélusqrt algorithm from \cite{BFLS20} cannot be applied here because the kernel of the isogenies and the point on which the evaluation is performed do not live in the same extension. 
\end{remark}

\begin{proposition}
\label{prop: complexity loop isogenies}
There exists a constant $C$ such that, at any index $i \leq  n-6$, the number of binary operations executed in each execution of the FOR loop in line~\ref{loop: isogenies} is upper bounded by $C \Phi(L_i) (n-i) \log p \llog(p)  M_{\PP} (\log p) M_\ZZ(\log p)$.
When $ n-6 < i < n$, the number of binary operations is smaller than $$ C \Phi(L_i) \llog(p)^2 M_{\PP} (\log p) M_\ZZ(\log p) .$$
\end{proposition}
\begin{proof}
For each execution of this loop, the number of iteration is exactly $\Phi(L_i)$. 
Arithmetic over quaternion orders and ideals such as intersection and right order computation can be performed in $O(M_\ZZ(\log p))$ (because the coefficients have size in $O(\log p)$ and these operations are simple linear algebra in dimension $4$). Then, the hash can be computed in $O(\log p^{1+\varepsilon})$. Thus, the total cost of the loop is $O( \Phi(L_i) \log p^{1+\varepsilon} )$ for quaternion operations for any $i$. This is negligible compared to other operations when $i< n$. 

The verification that $H(\O) \in \frakS$ and insertion in the hash map operation can be done in $O(1)$ with the appropriate hash-map structure. 

Then, there are the cost of operations over $\FF_p$ for the isogeny computation. To derive the total complexity, we apply \cref{prop: complexity and correction step isogeny computation,prop: complexity evaluation step loop isogenies}. Arithmetic over $\FF_p$ takes $O(M_\ZZ(\log p))$ binary operations. 
For any $i$, the kernel computation is negligible.
When $i \leq n-6$, the isogeny computation takes 
$$O( \Phi(L_i) \log p \llog(p) M_{\PP} (\log p) M_\ZZ(\log p)),$$ then the evaluations take $O( (n-i) \Phi(L_i) \log p M_{\PP} (\log p) M_\ZZ(\log p))$.

Similarly, when $n-6 < i< n$, this cost is replaced by $$O( \Phi(L_i) \llog(p)^2 M_{\PP} (\log p) M_\ZZ(\log p)).$$ 
\end{proof}

\begin{proposition}
\label{prop: complexity loop inner}
There exists a constant $C$ such that, at any index $i \leq n-6$, the number of binary operations executed in the FOR loop in line~\ref{loop: inner list} is upper bounded by $ C \prod_{j=1}^i \Phi(L_j) (n-i) \log p^{3 + \varepsilon} $.
When $n-6 <i < n$, it is upper bounded by $C \prod_{j=1}^i \Phi(L_j) (n-i) \log p^{2 + \varepsilon}$. 
When $i=n$, it is upper-bound by $C (\#\frakS  \log p^{2 + \varepsilon}+ p \log p^{1+\varepsilon} ) $. 
\end{proposition}
\begin{proof}
There are $\Phi(L_i)$ cyclic subgroups of order $L_i$. Thus at index $i< n$, the size of $\texttt{List}$ is $\prod_{j=1}^{i-1} \Phi(L_j)$ and the result follows directly from  \cref{prop: complexity loop isogenies}.

When $i = n$, we perform the quaternion computations (intersection, right order and computation of the hash value) for all $\prod_{j=1}^{n} \Phi(L_i)$ subgroups. Thus, since we have $\prod_{j=1}^{n} \Phi(L_j) =O(p)$ by \cref{claim: log factor basis,claim: number of maximal orders}, and the cost of quaternion operations is $O(\log p^{1 + \varepsilon})$ as in the proof of \cref{prop: complexity loop isogenies}, we get that the cost for quaternion operations is $O(p \log p^{1+\varepsilon} )$.
The isogeny computation is only performed when the right maximal order is contained in $\frakS$, thus, we can upper-bound the number of times where an isogeny is computed by $\# \frakS$. As for all the $i \geq n-6$, the cost of each $L_n$-isogeny computation is $O( \llog(p)^2 M_{\PP} (\log p) M_\ZZ(\log p))$ and this proves the result. 
\end{proof}

\begin{proposition}
The loop in line~\ref{loop: cofactors} can be executed in  $O(\#\frakS \log p^{2 + \varepsilon} + p \log p^{1+\varepsilon} )$ binary operations.
\end{proposition}
\begin{proof}
The total cost of the loop is directly obtained by summing over all $i$ the  bounds in \cref{prop: complexity loop inner}.  
We start by summing over all $i\leq n-6$. 
We get an upper-bound on the number of binary operation by $C \sum_{i=1}^{n-6} (n-i) \prod_{j=1}^i \Phi(L_j) \log p^{3 + \varepsilon}$. So we get $O( p \log p^{1+ \varepsilon})$ after applying \cref{prop: log factor basis critical bound}. 

For $n-6 < i < n$, we get the upper-bound 
$$C \sum_{i=n-5}^{n-1} (n-i) \prod_{j=1}^i \Phi(L_j) \log p^{2 + \varepsilon}.$$
Since by \cref{claim: log factor basis}, $\Phi(L_n) > C_0/2 \log p$, we have $\prod_{j=1}^i \Phi(L_j) \leq C' p/\log p$ for some constant $C'$ and any $i < n$. Thus, since there is a constant number of summands, we get the cost is in $O( p \log p^{1+\varepsilon})$ for those indices. 

Finally, at $i =n$, we can apply directly the bound from \cref{prop: complexity loop inner}. The final cost is $O(\#\frakS \log p^{2 + \varepsilon} + p \log p^{1+\varepsilon} )$.
\end{proof}

All the results above lead directly to the following theorem. 

\begin{theorem}
\label{thm: complexity}
Under \cref{,claim: number of maximal orders},
on input $p$ and $\frakS$,
\OrdersToJName{} can be executed in $$O(\#\frakS \log p^{2 + \varepsilon} + p \log p^{1+\varepsilon} )$$ binary operations.
The space requirement of \OrdersToJName{} is in $O(p \log p)$. 
\end{theorem}
\begin{proof}
    The cost of Step~\ref{step: good degree basis} is negligible before all the other computations. The result follows from all the propositions proven in this section. 

    For the space requirement, the proof is quite similar. 
    After each iteration of the loop for an index $i\leq n$, the list contains $ \prod_{j=1}^i \Phi(L_i)$ entries of the form $E,K,P_{i+1},\ldots,P_n,Q_{i+1},\ldots,Q_n$ where $E$ is an elliptic curve, $K$ is an ideal of norm $\prod_{j=1}^i L_i$ and $P_j,Q_j$ are points of order $L_i$. Thus, this can be represented in $O ( \prod_{j=1}^i \Phi(L_i) (\log p + (n-i) \log(p^2))$. Using \cref{prop: log factor basis critical bound} for $i \leq n-6$ we see that this takes $O(p \log p)$ space. With $\Phi(L_n) = \Theta(\log(p))$, we get that $\prod_{j=1}^i \Phi(L_i) = O(p /\log p )$ and so the memory required for indices $n-6 \leq i < n$ is also in $O(p \log p)$. Finally, at index $i=n$ the second term is zero and so we get $O(p \log p)$. 
 
    Finally, the space required by the list $M$ is $O(\# \frakS \log p) = O(p  \log p) $, and this concludes the proof. 
\end{proof}

\paragraph*{Good choice of primes.} 
The analysis we provided above for both \OrdersToJSmall{} and \OrdersToJName{} does not assume anything on the prime $p$. There are some "nice" choices of primes $p$ for which we could basically gain a factor $\log p$ over all elliptic curve operations by having all the required torsion point defined over $\FF_{p^2}$ (thus saving the cost of operations over big $\FF_p$-extensions). 
Since we are interested in a generic statement, we do not bother with these marginal gains. 

In the context of applying \OrdersToJName{} to the CRT method, this idea will have its importance in the concrete choice of primes $p_i$. However, due to the linear dependency in $p$, it does not appear possible to select all CRT primes among these ``nice" primes. And so the CRT complexity will depend on the worst-case complexity of our algorithm \OrdersToJName{}.

%% file: hilbert_comput.tex
\subsection{Computing the class polynomial modulo a non-split prime.}
\label{sec: mod class poly computation}

Let us fix some negative discriminant $D$ and a non-split prime $p$ in $\QQ(\sqrt{D})$ (and we assume further that $p$ is coprime to the conductor of $\frakO$). We write $h$ for the class number of $\frakO$. For simplicity, we assume below that the factorization of $D$ is known so this step is not part of our estimates. 

Our goal in this section is to explain how to compute the Hilbert class polynomial $H_D(X) \mod p$. This polynomial can be reconstructed from its roots that are $j$-invariants of some supersingular curves over $\FF_{p^2}$.   

When $p$ is split, the roots correspond to ordinary curves. In Sutherland's algorithm they are obtained in two main steps  algorithm \cite{sutherland2011computing}: start by identifying one root, and then enumerate through all the roots using the action of the class group $\Cl(\frakO)$ through isogenies of small degree. 

When $p$ is non-split, we have three main steps. We start to do something very similar to what is done for ordinary curves, but instead of working directly with the roots, we first use the Deuring correspondence, and identify the roots from the the corresponding maximal orders of the quaternion algebra $\QA$. The class group action of $\Cl(\frakO)$ can be realized much more efficiently in that case because all operations in $\QA$ are essentially linear algebra.

The Deuring correspondence states that the set of maximal order we obtain in this manner are isomorphic to the endomorphism rings of the elliptic curves we want to compute. Thus, we constitue a set $\frakS_D(p)$ of maximal order type and we can apply \OrdersToJBig{} on this set to compute the $j$-invariants we need. This execution constitutes our third step.  

The fourth and final step in our algorithm is common with the third step of the ordinary case: recover the polynomial $H_D$ from its roots. Note that this is done using standard polynomial arithmetic. 

This is described more precisely in \cref{alg: supersingular hilbert computation}. 

\begin{algorithm}[ht]
    \caption{$\SupersingularHilbertComputation{}(D,p)$}\label{alg: supersingular hilbert computation}
    \begin{algorithmic}[1]
    \Require {A discriminant $D <0$ and a prime $p$ such that $p$ is non-split in, and coprime to the conductor of,$\frakO$ the quadratic order of discriminant $D$. }
    \Ensure { $H_D(X) \mod p.$}
    \STATE \label{step: one max order} Find a maximal order $\O$ in $\QA$ with $\frakO \hookrightarrow \O$. 
    \STATE \label{step: all max order} Use the action of $\Cl(D)$ to find $\SDp$, the set of types of maximal orders in $\QA$ with an optimal embedding of $\frakO$. To each type $\O \in \SDp$ is associated a number $n_\O \geq 1$ counting the multiplicity of the type $\O$ in this enumeration. The set of these multiplicites is write $N_D(p)$.  
    \STATE \label{step: order to jinv} Compute $J_D(p) = \OrdersToJSmall{}(p,\SDp) $ (or use \OrdersToJBig{} if this is more efficient).
    \STATE Reconstruct $H(X) = \prod_{j_D(p) \in J} (X-j) \mod p$ from $\SDp, J_D(p)$ and $N_D(p)$. 
   \Return Return $H(X)$. 
    \end{algorithmic}
   \end{algorithm}

\begin{proposition}
    \label{prop: supersingular hilbert computation} 
    \SupersingularHilbertComputation{} is correct and it can be executed in $$O(\sqrt{|D|} ( \log(|D|)^{2+\varepsilon} \log(p)^{1 + \varepsilon} + \log(|D|)^\varepsilon \log(p)^{2 + \varepsilon}) + p \log(p)^{1 + \varepsilon}),$$
    binary operations and requires $O( (p + h(D)) \log p)$ memory when $\OrdersToJBig{}$ is used in Step~\ref{step: find enough types} (assuming GRH and \cref{claim: number of maximal orders}), and 

    $$O(\sqrt{|D|} (\log(|D|)^{\varepsilon} \log(p)^{5 + \varepsilon} + \log(|D|)^{2+\varepsilon} \log(p)^{1 + \varepsilon})),$$ binary operations and requires $O(h(D)\log(p))$ space with \OrdersToJSmall{} (under GRH and the heuristics from \cite{KLPT14}).
\end{proposition}
\begin{proof}

    Correctness follows from \cite[Theorem 3.4]{onuki2021oriented}. 
    We will now explain how each step can be performed and what is the best known complexity.  

    \noindent\textit{Step~\ref{step: one max order}.}
    This task has already been solved in the context of generating backdoor curves to the SIDH scheme \cite{quehen2021improved} and generating keys for the Séta encryption scheme \cite{Seta}. First, we need to solve a quadratic equation over $\QQ$ to find $a,b,c,d \in \QQ$ such that $\ZZ[a+ib + jc +kd]$ is the quadratic order of discriminant $D$. This can be done using Simon's algorithm \cite{simon2005quadratic} to solve quadratic forms in dimension 4. The complexity of Simon's algorithm is polynomial in the logarithm of the determinant once the factorization of the determinant is known and the size of the output is also logarithmic in the determinant. In our case, the quadratic form we consider is basically $b,c,d,e \mapsto (qb^2 + p(c^2 + qd^2) - e^2 D$ and its determinant is equal to $p^2 q^2 D 2^f$ for some small integer $f$. The computation of the full factorization is sub-exponential in $\log(pD)$.

    Now that $\theta = a+ib+jc+kd$ has been computed, we need to find a maximal order $\O$ containing it. Let us take $A$ as the smallest common denominator of $a,b,c,d$, we have $A = \polylog{pD}$. Then $A \theta \in \O_0$ where $\O_0$ is any maximal order containing the sub-order $\G{1,i,j,k}$. Since $A \theta \in \O_0$, the right order of the ideal $I = \O_0 A \theta + \O_0 C$ contains $\theta$. We can set $\O = \O_R(I)$ and $\O$ can be computed in $\polylog{p|D|}$. 
    Hence, this step can be performed in sub-exponential in $\log{p|D|}$ and is negligible compared to the rest of the computation.

\noindent\textit{Step~\ref{step: all max order}.} We go from one maximal order type to all maximal order types of interest by using the group action of the class group in a manner similar to what is used by Sutherland in \cite{sutherland2011computing}. For that, we use a polycyclic representation of $\Cl(D)$ as introduced in \cite[Section 5]{sutherland2011computing}. But, in our case, instead of isogeny computation, we can simply use arithmetic over quaternions through the action of ideals of the form $\O (\theta - \lambda) + \O \ell$ on the set of maximal orders containing $\theta$ which cover all maximal order types we need. 
Any group action computation for an ideal of norm $\ell$ takes $O(\log(\ell))$. Thus, using the same estimates than in \cite[Lemma 7]{sutherland2011computing} where we see that under GRH the biggest norm of an ideal involved a polycyclic representation is in $O(\log^2(D))$, we see that this part can be performed in $O(h \log(|D|)^\varepsilon) = O(\sqrt{|D|} \log(|D|)^\varepsilon)$.

We can hash (with the function introduced in \cref{sec: hash}) all the maximal order types obtained in this manner to create the set $\frakS_D(p)$ in $O( \sqrt{|D|} \log(|D|)^\varepsilon \- \log(p)^{1+\varepsilon})$.  

\noindent\textit{Step~\ref{step: order to jinv}.} This step consists simply in the execution of \OrdersToJBig{} on the set $\frakS_D(p)$ computed in Step~\ref{step: all max order}. Thus, by \cref{thm: complexity} and the estimates on $h$, the complexity of this step is $O(\sqrt{|D|} \log(|D|)^\varepsilon \log(p)^{2 + \varepsilon} + p \log(p)^{1+\varepsilon})$.
Alternatively, it is possible to use the \OrdersToJSmall{} algorithm and obtain a complexity of $O(\sqrt{|D|} \log(|D|)^\varepsilon \log(p)^{5 + \varepsilon})$. 

\noindent\textit{The reconstruction step.} 
The complexity of this step is $O(\sqrt{|D|} \log(|D|)^{2+\varepsilon} \- \log(p)^{1+\varepsilon})$ as was proven in \cite{sutherland2011computing}. 

\noindent\textit{The total complexity.} Putting together all the results above, we get the complexity result. 

\noindent\textit{Space complexity.} In terms of memory requirement, the polynomial reconstruction requires $O(h(D) \log p)$ space (since we only need to store two levels of the product tree at the same time). The set of maximal orders takes $O(h(D) \log p)$ space. For the variant with \OrdersToJBig{}, the results follows from \cref{thm: complexity}. When \OrdersToJSmall{} is used, the computation of each $j$-invariant is done sequentially, and so the space requirement is optimal and is $O(h(D)\log p)$. 
\qed{}
\end{proof}

\paragraph*{Which version is optimal?} We obtain two versions of \SupersingularHilbertComputation{} depending on weither we use \OrdersToJSmall{} or \OrdersToJBig{}. For a given value of $D$, the former has a better asymptotic complexity. However, for primes $p$ in $O (\sqrt{|D|} \log(|D|)^{4 + \varepsilon})$, the latter has a better complexity.

For a range of medium-sized primes, this algorithm will have the best known complexity. 
The cut-off with the CRT method (whose complexity is $O(|D|^{1+ \varepsilon})$) (as we will see below) happens for $p = O( 2^{|D|^{1/10}}) $. 

\subsection{ Application to the CRT method and comparison with existing method.}
\label{sec: class poly comparison}

In this section, we present and analyze the algorithm yielded by applying the CRT method from \cite{sutherland2011computing} on top of the \SupersingularHilbertComputation{} algorithm introduced in the previous section. The algorithm is pretty straightforward, and the only constraint on the selection of the small primes.


\begin{algorithm}[ht]
    \caption{$\HilbertComputation{}(D,p)$}\label{alg: hilbert computation}
    \begin{algorithmic}[1]
        \Require {A discriminant $D<0$ and a prime $p$.  }
        \Ensure { $H_D(X) \mod p.$}
    \STATE Set $B_D$ as the upper-bound on the coefficients of $H_D(X)$ over $\ZZ$. 
    \STATE Set $\primesO{} = \lbrace \rbrace$, $P=1$, $q=3$. 
    \WHILE { $P < B_D$ }
        \IF {$q$ is prime and $q$ is non-split in and coprime to the conductor of $\frakO$}
            \STATE $P \leftarrow q \cdot P$, \enspace $\primesell{} \leftarrow \primesell{} \cup \lbrace q \rbrace$. 
        \ENDIF
        \STATE $q \leftarrow q+2$
    \ENDWHILE
    \STATE Perform the precomputations for the explicit CRT mod $p$ using $\primesO{}$.   
    \FOR {$q \in \primesO$}
        \STATE $H_{\frakO,q}(X) \leftarrow \SupersingularHilbertComputation{}(D,q)$. 
        \STATE Update the CRT sums for each coefficient of $H_{\frakO,q}(X)$.  
    \ENDFOR 
    \STATE Perform the postcomputation for the explicit CRT to obtain $H_D(X) \in \FF_p[X]$. 
   \Return Return $H_D(X)$.
    \end{algorithmic}
   \end{algorithm}

\begin{proposition}
    \label{prop: modular computation}
    Under GRH and \cref{claim: number of maximal orders},
    the complexity of \HilbertComputation{} is $O(|D| \log^{3+ \varepsilon} |D|)$. The space requirement is $O(\sqrt{|D|} (\log^{2+\varepsilon} |D| + \log p )  )$
\end{proposition}
\begin{proof}
    Under GRH, we can lower bound the probability that a given prime is non-split in $\frakO$ by $1/2$. 
    Then, following the same reasoning as in \cite{sutherland2011computing}, it can be shown that we have $\# \primesO = O(\log B_D/\llog(B_D))$ and $\max_{q \in \primesO} q = O(\log(B_D))$. With the usual $\log B_D = O( \sqrt{|D|} \log^{1+\varepsilon} |D|)$ that holds under GRH, we get that we can take $O(\sqrt{|D|})$ primes with  $\max_{q \in \primesO} q =  O(\sqrt{|D|} \log (|D|))$. 
    By \cref{prop: supersingular hilbert computation}, the dominant complexity of each CRT computation is the cost of the polynomial reconstruction. Thus, following the complexity estimates from \cite{sutherland2011computing}, we obtain that the overall complexity estimate is $O(|D| \log(|D|)^{3 + \varepsilon})$.
    The space complexity follows from the space requirements computed in \cite{sutherland2011computing} and \cref{prop: supersingular hilbert computation}. 
\end{proof}



\paragraph*{Practical comparison with the CRT algorithm based on ordinary curves from \cite{sutherland2011computing}. } 
We can see that the complexity reported in \cref{prop: modular computation} is the same as Sutherland's algorithm from \cite{sutherland2011computing}. Even if the two algorithms have the same asymptotic complexity, they might not have the same practical efficiency. Below, we try to see which one could be faster in practice.  

First, note that the set of CRT primes are very different between the two methods. Since the probability that a given prime can be selected in $\primesO$ in our algorithm \HilbertComputation{} is $1/2$, the primes that we select are going to be much smaller than the ones used in Sutherland's algorithm.  

Now, we look at the concrete algorithmic steps. Let us start with the polynomial reconstruction step as it is the asymptotically dominant step. We remind that the concrete complexity of this step is $O(\sqrt{|D|} \log(|D|)^{2+ \varepsilon} \log(p)$. It is pretty similar in both algorithms and we argue that the practical cost should be roughly the same. This is not completely obvious since the primes will not have the same size in the two cases and the roots are defined over $\FF_p$ for ordinary curves against $\FF_{p^2}$ over supersingular curves. 
First, the size of the primes does not really matter because the product $\prod_{i=1}^n p_i$ have roughly the same size in the two cases and the complexity of the reconstruction is linear in $\log(p_i)$ for all $i$.
Second, since in the supersingular case, the Galois conjugate (by the action of the Frobenius) of a root of $H_D \mod p_i$ is also a root, by building the remainder tree from polynomials of the form $(X-j)(X-j^{p_i}) \in \FF_{p_i}[X]$, we see that we can make the entire computation over $\FF_{p_i}$ as in the ordinary case (and thus avoid the constant overhead brought by multiplications over $\FF_{p_i^2}$). 
We conclude from this brief reasoning that the reconstruction cost will be essentially the same in the two cases. 

Now, if we forget the reconstruction step, we see that using supersingular curve offers an asymptotic advantages. Indeed, in Steps~\ref{step: one max order},~\ref{step: all max order} and \ref{step: order to jinv} of our algorithm, the dominant step is the execution of \OrdersToJBig{} in Step~\ref{step: order to jinv}, which has a $O(|D| \log(|D|)^{2+\varepsilon})$ complexity (if we consider the executions over all primes $p_i \in \primes{}$ and we use $\log (p_i)  = O( \log(|D|) ) $). 
In particular, this is smaller than the $O(|D| \log(|D|)^{5/2 + \epsilon})$ that dominates that part of the computation in Sutherland's algorithm (corresponding to the computation of one curve with the correct endomorphism ring). 

This is the first reason that suggests that the supersingular case might be more efficient than the ordinary one, but this is not the main one. 
The main reason behind the practical speed-up we hope to obtain is that we can use smaller primes. Indeed, the expected maximum of our primes is in $O(\sqrt{|D|} \log(|D|))$ (against $O(|D|\log(|D|)^{1+\varepsilon}$ for ordinary curves). Moreover, we can take all the small primes that satisfy the reduosity condition. In particular, we will be able to use a good portion of primes significantly smaller than $\sqrt{|D|}$. 


We hope that the very small primes will give a nice improvement in practice because for these primes, some of the roots will have big multiplicities, which should help perform every steps more efficiently in practice (for example there will be less than $O(\sqrt{|D|})$ $j$-invariants to compute in that case) and it should improve the practical efficiency.  


Note that there is also a good potential for practical improvement by carefully selecting the primes in $\primes{}$ and choosing the good-degree basis used for each of those primes in order to minimize the degree of the extension required to compute the isogenies in \OrdersToJBig{}.
A selection is also performed in the algorithm of Sutherland to help improve the cost of finding one curve with the good cardinal, so it would be natural to do the same thing in our case. 

Even if \OrdersToJBig{} proves to be too slow to beat the version of Sutherland by using only supersingular primes, it is clear that it is worth considering an hybrid set of primes $\primes{}$ containing a mix of supersingular and ordinary primes to obtain the best efficiency as the computation will be definitely very fast for a lot of small non-split primes. 

\paragraph*{Further improvement: batching class polynomial computation.}
\OrdersToJBig{} can be easily modified to handle several sets $\frakS_1,\ldots,\frakS_k$ more efficiently than $k$ executions of \OrdersToJBig{} for each $\frakS_i$. 

Thus, if we have several discriminants $D_1,\ldots,D_k$, and a prime $p$ in $ \bigcap_{1 \leq i \leq k} \primes{k}$.  A good part of the computations performed to compute $H_{D_1},\ldots,H_{D_k} \mod p$ can be done at the same time at a reduced cost. 

Moreover, if some $H_{D_i}$ have some common roots, the common divisors could be constructed once and for all.  

Thus, our new method could be used to batch efficiently the computation of several class polynomial at the same time.

%% file: modular_comput.tex
\subsection{Direct modular polynomial computation with supersingular curves}
\label{sec: direct supersingular modular poly computation}

Let us take a prime level $\ell$ and a prime $p$.
To make a direct computation of $\Phi_\ell(X,Y) \mod p$ it suffices to identify $\ell+1$ distinct $j$-invariants and for each of those, to interpolate $\Phi_\ell(j,Y)$ from its roots that are the $\ell+1$ $ell$-isogenous $j$-invariants.  

The idea of our algorithm for modular polynomials follows the same principle as the class polynomials algorithm. There is a slight difference because modular polynomials are bivariate but it does not change the generic principle of the algorithm. Indeed the full polynomial $\Phi_\ell \mod p$ is interpolated from the $\Phi_\ell(j_i,Y)$ for $\ell+1$ $j$-invariants $(j_i)_{0 \leq i \leq \ell}$ of elliptic curves defined in $\overline{\FF_{p}}$. Each univariate polynomial $\Phi_\ell(j_i,Y)$ is reconstructed from its roots, that are the $j$-invariants of curves $\ell$-isogenous to $j_i$. 

In \SupersingularModularComputation{}{}, we compute the maximal orders that correspond to all the $j$-invariants that we need under the Deuring correspondence, and then, we apply \OrdersToJBig{} (or \OrdersToJSmall) to find the required $j$-invariants.  

\begin{algorithm}[ht]
    \caption{$\SupersingularModularComputation{}(p,\ell)$}\label{alg: supersingular modular computation}
    \begin{algorithmic}[1]
    \Require {A prime $ p$ and a prime $\ell$ such that $\lceil p/12 \rceil +1 > \ell$. }
    \Ensure { $\Phi_\ell(X,Y) \mod p.$}
    \STATE \label{step: find enough types} Compute a set $\frakS_0 = \lbrace \O_{1},\ldots,\O_{\ell+1} \rbrace$ of distinct maximal order types in $\QA$.
    \STATE Set $\frakS_1,\ldots, \frakS_{\ell+1} = \lbrace \rbrace,\ldots,\lbrace \rbrace$. 
    \FOR{ $i=1$ to $\ell+1$} 
        \STATE Compute $I_{i,1},\ldots,I_{i,\ell+1}$, the $\ell+1$ left $\O_{i}$-ideals of norm $\ell$.
        \STATE $\frakS_{i} \leftarrow \lbrace \O_R(I_{i,k}) | 1 \leq k \leq \ell+1 \rbrace$.  
    \ENDFOR
    \STATE $\frakS = \bigcup_{0 \leq i \leq \ell+1} \frakS_0$. 
    \STATE\label{step: modular j comput} Compute $J = \OrdersToJSmall{}(p,\frakS)$ (or use \OrdersToJBig{}$(p,\frakS)$ if this is more efficient).
    \STATE Divide $J$ as $J_0,J_1,\ldots J_{\ell+1}$ corresponding to $\frakS_0,\frakS_1,\ldots,\frakS_{\ell+1}$.
    \FOR{ $i=1$ to $\ell+1$}
        \STATE $P(j_{i},Y) \leftarrow \prod_{j \in J_i} (Y-j)$.
    \ENDFOR
    \STATE Reconstruct $P(X,Y) \mod p$ from the $P(j,Y)$ for $j \in J_0$. 
   \Return Return $P(X,Y)$. 
    \end{algorithmic}
   \end{algorithm}

\begin{proposition}
    \label{prop: supersingular modular computation} 
    \SupersingularModularComputation{} can be executed in 
    $$O(\ell^2 (\log \ell^{2 + \varepsilon} \log p^{1 + \varepsilon} + \log p^{2 + \varepsilon} ) + p \log p^{1 + \varepsilon})$$
    binary operations and requires $O( (\ell^2 + p) \log p)$ when \OrdersToJBig{} is use in Step~\ref{step: modular j comput} (assuming GRH and \cref{claim: number of maximal orders}), and
    $$O(\ell^2 (\log p^{5 + \varepsilon} + \log \ell^{2 + \varepsilon} \log p^{1+ \varepsilon}) )$$
    binary operations and requires $O(\ell^2 \log p)$ space with \OrdersToJSmall{} (under GRH and the heuristics from \cite{KLPT14}).  
\end{proposition}
\begin{proof}

    With the condition on the respective size of $\ell$ and $p$, we know there are enough supersingular $j$-invariants defined over $\FF_{p^2}$. 
    We now briefly explain how each of the steps can be performed and what are their complexities.
    
    By \cref{claim: number of maximal orders}, there is a value $L =O(p)$ such that every supersingular $j$-invariant are $L$-isogenous to a given $j$-invariant $j_0$. Thus, under the Deuring correspondence, the $\ell+1$ maximal order types can be computed in $O(\ell \log(p)^{1+ \varepsilon})$ by starting from a canonical maximal order type $\O_0 \subset \QA$ (for which there exists formulas in \cite{KLPT14} for any $p$). Thus, the complexity of Step~\ref{step: find enough types} is $O(\ell \log(p)^{1 + \varepsilon})$.  
    
    The computation of an $\ell$-ideal and of its right order can be done in $O(\log (\ell p))$. Thus, the set of maximal order types $\frakS$ can be computed in $O(\ell^2 (\log p + \log \ell))$.  

    As show in \cite{broker2012modular}, the interpolation of all the $\Phi_\ell(j_i,Y)$, and the final interpolation of $\Phi_\ell(X,Y)$ take $O ( \ell^2 \log \ell^{2 + \varepsilon} \log p^{1 + \varepsilon})$ . 

    The final complexity statement follows from the results shown previously on the complexities of \OrdersToJBig{} and \OrdersToJSmall{}. 

    The space complexity of the polynomial reconstruction is $O(\ell^2 \log p )$ and the results follow from the space complexites of \OrdersToJBig{} and \OrdersToJSmall{}. 
\end{proof}



\paragraph{Which version is better ?}
Thus, we see that the first algorithm based on \OrdersToJName{} will be better for small values of $p$. We can estimate a cut-off for a value of $p$ in  $O (\ell^{2+\varepsilon})$. In that range of primes, our algorithm with \OrdersToJName{} has the best known generic complexity to compute $\Phi_\ell \mod p$. 
For primes bigger than that, it is better to use the variant with \OrdersToJSmall{} to avoid the quasi-linear dependency in $p$. For a range of medium-sized primes, this algorithm will have the best known complexity.


\paragraph*{Space complexity.} In terms of memory requirement our two algorithms are optimal and require $O(\ell^2\log p)$. 

\subsection{CRT methods to compute modular polynomials from supersingular curves. }
\label{sec: CRT modular poly computation}

In this section, we analyse the benefit of using supersingular curves in the CRT method and compare it with the algorithm described by Bröker, Lauter and Sutherland in \cite{broker2012modular}. We refer the reader to the algorithm outlined in \cref{sec: technical overview}. Since the CRT is typically based on a lot of very small primes, we use the variant with \OrdersToJName{}.

\begin{algorithm}[ht]
    \caption{$\ModularComputation{}(p,\ell)$}\label{alg: modular computation}
    \begin{algorithmic}[1]
    \Require {A prime $p$, $j \in \FF_p$, a prime $\ell$. }
    \Ensure { $\Phi_\ell(X,Y) \mod p$.} 
    \STATE Let $\overline{j}$ be the integer in $[0,p-1]$ equal to $j \mod p$. 
    \STATE Set $B_\ell = 2^{6 \ell \log \ell + 18 \ell + \log (\ell+2)} $. 
    \STATE Set $\primesell{} = \lbrace \rbrace$, $P=1$, $q= 12 \ell $. 
    \WHILE { $P < B_\ell$ }
        \IF {$q$ is prime}
            \STATE $P \leftarrow q \cdot P$, \enspace $\primesell{} \leftarrow \primesell{} \cup \lbrace q \rbrace$. 
        \ENDIF
        \STATE $q \leftarrow q+2$
    \ENDWHILE
    \STATE Perform the precomputations for the explicit CRT mod $p$ using $\primesell{}$.   
    \FOR {$q \in \primesell$}
        \STATE $P_q(X,Y) \leftarrow \SupersingularModularComputation{}(q,\ell)$. 
        \STATE Update the CRT sums for each coefficient of $P_q(X,Y)$.  
    \ENDFOR 
    \STATE Perform the postcomputation for the explicit CRT to obtain $P(X,Y) \in \FF_p[X,X]$. 
   \Return Return $P(X,Y)$.
    \end{algorithmic}
   \end{algorithm}

\begin{proposition}
    \label{prop: modular computation}
    The complexity of \ModularComputation{} is $O(\ell^3 \log \ell^{3+ \varepsilon})$. The space requirement is $O(\ell^2 \log p)$. 
\end{proposition}
\begin{proof} 
    Following the same reasoning as in \cite{sutherland2011computing}, it can be shown that we have $\# \primesell = O(B_\ell/\log(B_\ell))$ and $\max_{q \in \primesell} q = O(\log(B_\ell))$. The choice of $B_\ell$ for the bound on the coefficients of $\Phi_\ell$ holds under GRH. We get that we can take $O(\ell)$ primes with  $\max_{q \in \primesell} q =  O(\ell \log (\ell))$.  

    By \cref{prop: supersingular modular computation}, for the complexity of \SupersingularModularComputation{} with \OrdersToJBig{}, and the complexity of the CRT steps described in \cite{broker2012modular}, we see that the complexity of each CRT computation takes $O(\ell^2 \log(\ell)^{3+\varepsilon})$. Thus, we obtain that the overall complexity estimate is $O(\ell^3 \log(\ell)^{3 + \varepsilon})$.
    The space complexity follows in a similar manner. 
\end{proof}

\paragraph*{On the cost of elliptic curve operations.}
The complexity of our algorithm is dominated by the polynomial reconstruction step. Note that this is also the case for the BLS algorithm. 

However, unlike BLS, the rest of our computation is much more efficient. Indeed, if we look at the cost of the elliptic curve operation (the cost of all the elliptic curve operations inside the calls to \SupersingularModularComputation{}), we see that the global complexity is only quadratic in $\ell$. Indeed, we use primes $q$ of size $O(\ell^{1+ \varepsilon})$ for which there exists only $O(\ell^{1+\varepsilon})$ supersingular j-invariants. This means that we can simply compute them all with \OrdersToJBig{} in quasi-linear time in $\ell$, thus yielding a quadratic complexity over all CRT primes. 
This is why we expect our algorithm to produce a practical speed-up over the BLS method.

\paragraph*{Batching the computation.}
Similarly to the Hilbert polynomial case, the set of small primes can be reused in the computations over various $\ell$. In fact, the situation is even better in the modular case, because, apart from size, there are no restrictions on the primes. Thus, if we want to compute $\Phi_{\ell_i}$ for primes $\ell_1,\ell_2,\ldots, \ell_k$ of the same size, we will be able to use the same exact set $\primes{}$ for all the computations. Thus, only the polynomial reconstruction phase will be specific to each $\ell_i$, and the rest needs to be done only once.

\subsection{Comparison between different existing methods}
\label{sec: modular comparison}

Here, we compare the complexity of various known methods to compute modular polynomials. Note that some of these methods are proven under some quite ad-hoc heuristics (ours in particular) while some others are proven rigorously. However, all the heuristics used in the various method are plausile and have been verified experimentally. Thus, we don't expect that these heuristics should be problematic for a practical comparison.  

Here is the list of algorithms to consider:

\begin{enumerate}
    \item Our direct algorithm \SupersingularModularComputation{}.
    \item The direct algorithm from Robert \cite{robert2022some}. 
    \item The CRT algorithm from BLS \cite{broker2012modular}. 
    \item Our CRT algorithm \ModularComputation{}.
    \item The CRT algorithm from Robert \cite{robert2022some}. 
\end{enumerate}

The two direct algorithms are the only known algorithms for generic primes whose complexity is only quadratic in $\ell$. We have stated the complexity of the two variants of \SupersingularModularComputation{} in \cref{prop: supersingular modular computation}. The complexity of Robert's direct algorithm is $O(\ell^2 \log p \log^{2+ 3 u} \ell)$ for some parameter $u$ that might be $2$ or $4$ (it's not really clear which one). 

On the other hand, the complexity of the three CRT algorithms is the same: it is cubic in $\ell$ and independand of $p$. BLS and ours are optimal in memory, while Robert's is not. 

We will compare these various complexities by makin the value of $p$ vary while $\ell$ remains fixed. 

For small values of $p$, \SupersingularModularComputation{} with \OrdersToJBig{} will be the best algorithm as noted in the end of \cref{sec: direct supersingular modular poly computation}. However, due to the linear complexity of $p$, this algorithm will quickly be outperformed by the other algorithms as $p$ grows. As soon as $\ell = o (\sqrt{p})$, it will be better to use \SupersingularModularComputation{} with \OrdersToJSmall{}. The $\log \ell$ factor is only quadratic in the complexity of \SupersingularModularComputation{} so it will remain better than Robert's direct algorithm when $p$ remains not too big. However, since the $\log p$ factor has an exponent of $5$ in \SupersingularModularComputation{} while Robert's algorithm is linear in $\log p$, the latter will eventually outperform our direct algorithm. The exact cross-point will depend on the value of $u$. 

Finally, due to the fact that the CRT methods have a running time independant of $p$, they will end up being more efficient than the direct algorithms when $\ell = o(\log p)$. 

The practical efficiency of Robert's algorithm is hard to estimate, but it might not be competitive with the other two due to its much bigger memory requirement (which ends being quite problematic as modular polynomial are huge). 
We have already argued that we expect our method to outperform the method from BLS in practice.   


